\documentclass[10pt]{article}
\usepackage{amsmath} 

\usepackage[T1]{fontenc}
\usepackage[dvips]{color}
\usepackage{subfigure}
\usepackage{graphicx}
\usepackage{boxedminipage}
\definecolor{dgreen}{rgb}{0,.8,.3}
\definecolor{lblue}{rgb}{.2,.3,.7}

\newtheorem{assumption}{Assumption}
\newtheorem{theorem}{Theorem}

\newtheorem{lemma}{Lemma}
\newtheorem{proposition}{Proposition}
\newtheorem{corollary}{Corollary}

\newtheorem{remark}{Remark}

\numberwithin{equation}{section}
\numberwithin{lemma}{section}
\numberwithin{theorem}{section}

\newcommand{\beq}{\begin{equation}}
\newcommand{\eeq}{\end{equation}}
\newcommand{\beqa}{\begin{eqnarray}}
\newcommand{\eeqa}{\end{eqnarray}}
\newcommand{\beqas}{\begin{eqnarray*}}
\newcommand{\eeqas}{\end{eqnarray*}}
\newcommand{\ba}{\begin{array}}
\newcommand{\ea}{\end{array}}
\newcommand{\bi}{\begin{itemize}}
\newcommand{\ei}{\end{itemize}}
\newcommand{\gap}{\hspace*{1em}}
\newcommand{\vgap}{\hspace*{1em}}

\newcommand{\nn}{\nonumber}

\def\bE{{\bf E}}

\def\blambda{{\bar \lambda}}
\def\bK{{\bar K}}

\def\cS{{\cal S}}

\def\dom{{\rm dom}}
\def\Diag{{\rm Diag}}

\def\hc{{\hat c}}

\def\hd{{\hat d}}

\def\hlambda{{\hat \lambda}}

\def\hv{{\hat v}}
\def\i{{\iota}}

\def\P{{\bf P}}

\def\tc{{\tilde c}}

\def\td{{\tilde d}}
\def\tf{{\tilde f}}

\def\tH{{\tilde H}}
\def\tlambda{{\tilde \lambda}}

\def\tx{{\tilde x}}

\title{Randomized block proximal damped Newton method for composite self-concordant minimization}

\author{
 Zhaosong Lu
 \thanks{Department of Mathematics, Simon Fraser University, Canada
(Email: {\tt zhaosong@sfu.ca}). This author was supported in part by NSERC Discovery Grant.}
}

\date{June 30, 2016}

\begin{document}

\maketitle

\begin{abstract}

In this paper we consider the composite self-concordant (CSC) minimization
problem, which minimizes the sum of a self-concordant function $f$ and a (possibly
nonsmooth) proper closed convex function $g$. The CSC minimization is the
cornerstone of the path-following interior point methods for solving a broad class of
convex optimization problems. It has also found numerous applications in machine
learning. The proximal damped Newton (PDN) methods have been well studied in the
literature for solving this problem that enjoy a nice iteration complexity. Given
that at each iteration these methods typically require evaluating or accessing the
Hessian of $f$ and also need to solve a proximal Newton
subproblem, the cost per iteration can be prohibitively high when applied to large-scale
problems. Inspired by the recent success of block coordinate descent methods, we
propose a randomized block proximal damped Newton (RBPDN) method for solving
the CSC minimization. Compared to the PDN methods, the computational cost per
iteration of RBPDN is usually significantly lower. The computational experiment on
a class of regularized logistic regression problems demonstrate that RBPDN is indeed
promising in solving large-scale CSC minimization problems. The convergence of RBPDN is also analyzed in the paper. In particular, we show that RBPDN is globally convergent
when $g$ is Lipschitz continuous. It is also shown that RBPDN enjoys a local linear
convergence. Moreover, we show that for a class of $g$ including the case where $g$ is smooth (but not necessarily self-concordant) and $\nabla g$ is Lipschitz continuous in its domain, RBPDN enjoys a global linear convergence.
 As a striking consequence, it shows that the classical damped Newton methods \cite{Ne04,ZhXi15-1} and the PDN \cite{TrKyCe15} for such $g$ are globally linearly
convergent, which was previously unknown in the literature. Moreover,
this result can be used to sharpen the existing iteration complexity of these methods.

\end{abstract}

\section{Introduction}

In this paper we are interested in the composite self-concordant minimization:
\beq \label{self-concord}
F^* = \min\limits_x \left\{F(x) := f(x) + g(x)\right\},
\eeq
where $f:\Re^N \to \bar\Re:=\Re\cup \{\infty\}$ is a self-concordant function with parameter $M_f \ge 0$ and $g:\Re^N \to \bar\Re$ is a (possibly nonsmooth) proper closed convex function. Specifically, by the standard definition of a self-concordant function (e.g., see \cite{NeNe94,Ne04}),  $f$ is convex and three times continuously differentiable in its domain denoted by $\dom(f)$, and moreover,
\[
|\psi'''(0)| \le M_f (\psi''(0))^{3/2}
\]
holds for every $x\in\dom(f)$ and $u\in\Re^N$, where $\psi(t)=f(x+tu)$ for any
$t\in\Re$. In addition, $f$ is called a standard self-concordant function if $M_f=2$.

It is well-known that problem \eqref{self-concord} with $g=0$ is  the
cornerstone of the path-following interior point methods for solving a broad class of
convex optimization problems. Indeed,  in the seminal work by Nesterov and Nemirovski
\cite{NeNe94}, many convex optimization problems can be recast into the problem:
\beq \label{convex-opt}
\min\limits_{x\in\Omega} \langle c,x \rangle,
\eeq
where $c\in\Re^N$, $\Omega \subseteq \Re^N$ is a closed convex set equipped
with a self-concordant barrier function $B$, and $\langle \cdot, \cdot \rangle$ denotes
the standard inner product.  It has been shown that an approximate
solution of problem \eqref{convex-opt} can be found by solving approximately a
sequence of barrier problems:
\[
\min\limits_x \left\{f_t(x) := \langle c,x \rangle + t B(x)\right\},
\]
where $t>0$ is updated with a suitable scheme. Clearly, these
barrier problems are a special case of \eqref{self-concord} with $f=f_t$ and $g=0$.


Recently, Tran-Dinh et al.\ \cite{TrKyCe14} extended the aforementioned
path-following scheme to solve the problem
\[
\min\limits_{x\in\Omega} g(x),
\]
where $g$ and $\Omega$ are defined as above. They showed that
an approximate solution of this problem can be obtained by solving approximately a
sequence of composite barrier problems:
\[
\min\limits_x t B(x) + g(x),
\]
where $t>0$ is suitably updated. These problems are also a special case of  \eqref{self-concord} with $f=t B$.

In addition, numerous models in machine learning are also
a special case of  \eqref{self-concord}. For example, in the context of supervised
learning, each sample is recorded as $(w,y)$, where $w\in\Re^N$ is a sample feature vector
and $y\in\Re$ is usually a target response or a binary (+1 or -1)  label. A loss
function $\phi(x;w,y)$ is typically associated with each $(w,y)$. Some popular loss functions
include, but are not limited to:
\bi
\item squared loss: $\phi(x;w,y) = (y-\langle w, x\rangle)^2$;
\item logistic loss: $\phi(x;w,y) = \log(1+\exp(-y\langle w, x\rangle)$.
\ei
A linear predictor is often estimated by solving the empirical
risk minimization model:
\[
\min\limits_x \underbrace{\frac1m \sum^m_{i=1} \phi(x;w^i,y_i) + \frac{\mu}{2} \|x\|^2}_{\tf(x)} +   g(x),
\]
where $m$ is the sample size and $g$ is a regularizer such as $\ell_1$ norm. For stability purpose,  the regularization term $\mu \|x\|^2/2$, where $\mu>0$ and $\|\cdot\|$ is the Euclidean norm, is often included to make the model strongly convex  (e.g., see  \cite{ZhXi15-1,ZhXi15-2}). It is easy to observe that when $\phi$ is the squared
loss, the associated $\tf$ is self-concordant with parameter $M_{\tf}=0$. In addition,  when
$\phi$ is the logistic loss, $y_i \in \{-1,1\}$ for all $i$ and $\mu>0$, Zhang and Xiao \cite{ZhXi15-1,ZhXi15-2} showed that the associated $\tf$ is self-concordant with parameter $M_{\tf}=R/\sqrt{\mu}$,
where $R=\max_i \|w^i\|$. Besides, they proved that the associated $\tf$
for a general class of loss functions $\phi$ is self-concordant, which includes a
smoothed hinge loss.

As another example, the graphical model is often used in statistics to estimate the conditional independence of a set of random variables (e.g., see \cite{YuLi07,DaBaEl08,FrHaTi08,Lu10}), which is in the form of:
\[
\min\limits_{X \in \cS^N_{++}} \langle S, X \rangle - \log\det(X) + \rho \sum_{i \neq j}|X_{ij}|,
\]
where $\rho>0$, $S$ is a sample covariance matrix, and $\cS^N_{++}$ is the set of
$N \times N$ positive definite matrices. Given that  $-\log\det(X)$ is a self-concordant function in $\cS^N_{++}$  (e.g., see \cite{Ne04}), it is clear to see that the
graphical model is also a special case of \eqref{self-concord}.

When $g=0$, problem \eqref{self-concord} can be solved by a damped
Newton (DN) method or a mixture of DN and Newton methods (e.g.,
see \cite[Section 4.1.5]{Ne04}).
To motivate our study, we now briefly review these methods for solving \eqref{self-concord} with $g=0$.
In particular, given an initial point $x^0 \in \dom(F)$, the DN method updates the iterates according to
\[
x^{k+1} = x^k + \frac{d^k}{1+\lambda_k}, \quad\quad \forall k \ge 0,
\]
where $d^k$ is the Newton direction and $\lambda_k$ is the local norm of $d^k$ at
$x^k$, which are given by:
\beq \label{dk}
d^k = -(\nabla^2 f(x^k))^{-1}\nabla f(x^k), \quad\quad
\lambda_k =\sqrt{(d^k)^T \nabla^2 f(x^k) d^k}.
\eeq
The mixture of DN and Newton first applies DN and then switches to the standard Newton
method (i.e., setting the step length to $1$) once an iterate is sufficiently close to the optimal solution.
The discussion in \cite[Section 4.1.5]{Ne04}  has a direct implication that both DN and the mixture of
DN and Newton find an approximate solution $x^k$ satisfying  $\lambda_k \le \epsilon$ in at most
\[
O\left(F(x^0)-F^*+\log\log \epsilon^{-1}\right)
\]
 iterations. This complexity can be obtained by considering two phases of these
methods. The first phase consists of the iterations executed by DN for generating
a point lying in a certain neighborhood of the optimal solution in which the local
quadratic convergence of DN or the standard Newton method is ensured to occur,
while the second phase consists of the rest of the iterations. Indeed,  $O\left(F(x^0)-F^*\right)$ and $O(\log\log \epsilon^{-1})$ are an estimate of the number of iterations of these two phases, respectively.

Recently, Zhang and Xiao \cite{ZhXi15-1,ZhXi15-2} proposed an inexact damped Newton (IDN) method for solving
\eqref{self-concord} with $g=0$. Their  method is almost identical to DN except that the search direction $d^k$ defined
in \eqref{dk} is inexactly computed by solving approximately the linear system
\[
\nabla^2 f(x^k) d = -\nabla f(x^k).
\]
By controlling suitably the inexactness on $d^k$ and considering the similar two phases
as above, they showed that IDN can find an approximate solution $x^k$ satisfying  $F(x^k)-F^* \le \epsilon$ in at most
\beq \label{complexity-idn}
O\left(F(x^0)-F^*+\log \epsilon^{-1}\right)
\eeq
iterations.

In addition, Tran-Dinh et al.\ \cite{TrKyCe15} recently proposed a proximal damped
 Newton (PDN) method and a proximal Newton method for solving \eqref{self-concord}.
These methods are almost the same as the aforementioned DN and the mixture of DN
and Newton except that $d^k$  is chosen as the following proximal Newton direction:
\beq  \label{prox-dir}
d^k = \arg\min\limits_d \left\{ f(x^k) + \langle \nabla f(x^k), d\rangle + \frac12 \langle d,\nabla^2 f(x^k) d\rangle + g(x^k+d) \right\}.
\eeq
It has essentially been shown in  \cite[Theorems 6, 7]{TrKyCe15} that the PDN and the proximal Newton method
can find an approximate solution $x^k$ satisfying $\lambda_k \le \epsilon$ in at most
\beq \label{complexity-pdn}
O\left(F(x^0)-F^*+\log\log \epsilon^{-1}\right)
\eeq
iterations, where $\lambda_k = \sqrt{(d^k)^T \nabla^2 f(x^k) d^k}$. This complexity was derived similarly as for the DN and the mixture of DN and Newton by considering the two phases mentioned above.

Besides, proximal gradient type methods and proximal Newton type methods have
been proposed in the literature for solving a class of composite minimization problems in
the form of  \eqref{self-concord} (e.g., see  \cite{BeTe09,Ne07,HsSuDhRa11,BeFa12,LeSuSa14}).   At each iteration, proximal gradient
type methods require the gradient of $f$ while proximal Newton type methods need to
access the Hessian of $f$ or its approximation. Though the proximal Newton type
methods \cite{BeFa12,LeSuSa14} are applicable to solve \eqref{self-concord}, they
typically require a linear search procedure to determine a suitable step length, which
may be expensive for solving large-scale problems. In this paper we are only interested in
a line-search free method for solving problem \eqref{self-concord}.

It is known from \cite{TrKyCe15} that PDN has a better iteration complexity than the
accelerated proximal gradient methods \cite{BeTe09,Ne07}. The cost per iteration of
PDN is, however, generally much higher because it computes the search direction $d^k$
according to \eqref{prox-dir} that involves $\nabla^2 f(x^k)$.  This can bring an
enormous challenge to PDN for solving large-scale problems.  Inspired by the recent success of block coordinate descent methods, block proximal gradient methods and
block quasi-Newton type methods (e.g., see  \cite{BeTe13, ChHsLi08, FeRi15, HoWa, LeSi13, LeLe10, LiLuXi15, LiWr14, LuXi13, LuXi15, Ne12, PaNe15, QuRiTaFe16, RiTa14, ShZh13,TsYu09, WeGoSc12, Wr12})  for solving large-scale problems,  we propose a randomized block proximal damped Newton (RBPDN) method for solving \eqref{self-concord} with
\beq \label{gx}
g(x) = \sum_{i=1}^n g_i(x_i) ,
\eeq
where each $x_i$ denotes a subvector of~$x$ with dimension~$N_i$, $\{x_i :i=1,\ldots,n\}$ form a partition of the components of~$x$, and each~$g_i: \Re^{N_i} \to \bar\Re$ is a proper closed convex function. Briefly speaking, suppose that
$p_1, \ldots, p_n>0$ are a set of probabilities such that $\sum_i p_i=1$. Given a current
iterate $x^k$, we randomly choose $\i\in \{1,\ldots,n\}$ with probability $p_{\i}$. The
next iterate $x^{k+1}$ is obtained by setting  $x^{k+1}_j = x^k_j$ for $j \neq \i$ and
\[
x^{k+1}_{\i} = x^k_{\i} + \frac{d_{\i}(x^k)}{1+\lambda_{\i}(x^k)},
\]
where $d_{\i}(x^k)$ is an approximate solution to the subproblem
\beq \label{rbpdn-subprob}
\min_{d_\i} \left\{f(x^k) + \langle \nabla_\i f(x^k), d_\i\rangle + \frac12 \langle d_\i,
\nabla^2_{\i\i} f(x^k), d_\i \rangle +g_\i(x^k_\i+d_\i)\right\},
\eeq
$\lambda_\i(x^k) = \sqrt{\langle d_\i(x^k), \nabla^2_{\i\i}f(x^k) d_\i(x^k)\rangle}$, and $ \nabla_\i f(x^k)$ and $\nabla^2_{\i\i}f(x^k)$ are respectively the subvector and  the submatrix of $\nabla f(x^k)$ and $\nabla^2 f(x^k)$ corresponding to $x_\i$.

In contrast with the (full) PDN \cite{TrKyCe15}, the cost per iteration of RBPDN can be
considerably lower because: (i) only the {\it submatrix} $\nabla^2_{\i\i} f(x^k)$ rather
than the {\it full} $\nabla^2 f(x^k)$ needs to be accessed and/or evaluated; and (ii) the
dimension of subproblem \eqref{rbpdn-subprob} is much smaller than that of
\eqref{prox-dir} and thus the computational cost for solving \eqref{rbpdn-subprob} can
also be substantially lower. In addition, compared to the randomized block
accelerated proximal gradient (RBAPG) method \cite{FeRi15,LiLuXi15},
RBPDN utilizes the entire curvature information in the random subspace
(i.e., $\nabla^2_{\i\i} f(x^k)$) while RBAPG only uses the partial curvature information, particularly, the extreme eigenvalues of
$\nabla^2_{\i\i} f(x^k)$.  It is thus expected that RBPDN takes less number of
iterations than RBAPG for finding an approximate solution of similar quality, which is
indeed demonstrated in our numerical experiments.
Overall, RBPDN can be much faster than RBAPG, provided that the subproblem \eqref{rbpdn-subprob} is efficiently solved.

The convergence of RBPDN is analyzed in this paper. In particular, we show that
when $g$ is Lipschitz continuous in
\beq \label{S0}
\cS(x^0) := \{x: F(x) \le F(x^0)\},
\eeq
RBPDN is globally convergent, that is, $\bE[F(x^k)] \to F^*$ as $k\to\infty$. It is also
shown that RBPDN enjoys a local linear convergence. Moreover, we show that for a class of $g$ including the case where $g$ is smooth (but not necessarily self-concordant) and $\nabla g$ is Lipschitz continuous in $\cS(x^0)$, RBPDN
enjoys a global linear convergence, that is, there exists some $q\in (0,1)$ such that
\[
\bE[F(x^k)-F^*] \le q^k (F(x^0)-F^*),  \quad\quad \forall k \ge 0,
\]
Notice that the DN \cite{Ne04} and PDN \cite{TrKyCe15} are a special case of RBPDN
with $n=1$. As a striking consequence, it follows that they are globally linearly
convergent for such $g$, which was previously unknown in the literature.
Moreover, this result can be used to sharpen the existing iteration complexity of the
first phase of DN \cite{Ne04}, IDN \cite{ZhXi15-1}, PDN \cite{TrKyCe15},
the proximal Newton method \cite{TrKyCe15} and the mixture of DN and Newton \cite{Ne04}.

The rest of this paper is organized as follows. In Subsection \ref{notation}, we present some assumption, notation and also some known facts. We propose in Section \ref{alg}
a RBPDN method for solving problem \eqref{self-concord}  in which $g$ is in the form
of  \eqref{gx}. In Section \ref{tech}, we provide some technical preliminaries.
The convergence analysis of RBPDN is given in Section \ref{converge}. Numerical
results are presented in Section \ref{res}.

\subsection{Assumption, notation and facts} \label{notation}

Throughout this paper, we make the following assumption for problem \eqref{self-concord}.

\begin{assumption} \label{assump}
\bi
\item[(i)]
$f$ is a standard self-concordant function\footnote{It follows from \cite[Corollary 4.1.2]{Ne04} that if $f$ is self-concordant with parameter $M_f$, then $\frac{M^2_f}{4} f$ is
a standard self-concordant function. Therefore, problem \eqref{self-concord} can be rescaled into an equivalent problem for which Assumption \ref{assump} (i) holds.} and $g$ is in the form of \eqref{gx}.
\item[(ii)]
$\nabla^2 f$ is continuous and positive definite in the domain of $F$.
\item[(iii)] Problem \eqref{self-concord} has a unique optimal solution $x^*$.
\ei
\end{assumption}

Let $\Re^N$ denote the Euclidean space of dimension $N$ that is
equipped with the standard inner product $\langle \cdot,\cdot \rangle$.  For every $x\in\Re^N$, let $x_i$ denote a subvector
of~$x$ with dimension~$N_i$, where $\{x_i :i=1,\ldots,n\}$ form a particular partition of
the components of~$x$.

$\|\cdot\|$ denotes the Euclidean norm  of a vector or the spectral norm of a matrix.
The local norm and its dual norm at any $x\in\dom(f)$ are given by
\[
\|u\|_x := \sqrt{\langle u, \nabla^2 f(x) u\rangle}, \quad \|v\|^*_x := \sqrt{\langle v, (\nabla^2 f(x))^{-1} v\rangle}, \quad\quad \forall u, v \in \Re^N.
\]
 It is easy to see that
\beq \label{uv-product1}
|\langle u, v \rangle| \le \|u\|_x \cdot \|v\|^*_x, \quad\quad \forall u, v \in \Re^N.
\eeq
For  any $i\in \{1,\ldots,n\}$, let $\nabla^2_{ii} f(x)$ denote the submatrix of $\nabla^2 f(x)$ corresponding to the subvector $x_i$. The local norm and its dual norm of $x$  restricted to the subspace of $x_i$ are defined as
\beq \label{xi-norm}
\|y\|_{x_i} := \sqrt{\langle y, \nabla^2_{ii} f(x) y\rangle}, \quad \|z\|^*_{x_i} := \sqrt{\langle z, (\nabla^2_{ii} f(x))^{-1} z\rangle}, \quad\quad \forall y,z \in \Re^{N_i}.
\eeq
In addition, for any symmetric positive definite matrix $M$, the weighted norm and its
dual norm associated with $M$ are defined as
\beq \label{M-norm}
\|u\|_M := \sqrt{\langle u, Mu \rangle}, \quad\quad \|v\|^*_M := \sqrt{\langle v, M^{-1} v \rangle}.
\eeq
It is clear that
\beq \label{uv-product2}
|\langle u, v \rangle| \le \|u\|_M \cdot \|v\|^*_M.
\eeq

The following two functions have played a crucial role in studying some properties of a standard self-concordant function (e.g., see \cite{Ne04}):
\beq \label{omega}
\omega(t) = t-\ln(1+t),\quad\quad \omega_*(t) = -t-\ln(1-t).
\eeq
 It is not hard to observe that $\omega(t) \ge 0$ for
all $t>-1$ and $\omega_*(t) \ge 0$ for every $t<1$, and moreover,  $\omega$ and
$\omega_*$ are strictly increasing in $[0,\infty)$ and $[0,1)$, respectively. In addition, they are conjugate of each other, which implies that for any $t \ge 0$ and $\tau\in[0,1)$,
\beq \label{conj}
 \omega(t) = t \omega'(t)- \omega_*(\omega'(t)), \quad\quad \omega(t) + \omega_*(\tau) \ge \tau t
\eeq
(e.g., see \cite[Lemma 4.1.4]{Ne04}).

It is known from  \cite[Theorems 4.1.7, 4.1.8]{Ne04}) that $f$ satisfies:
\beqa
f(y) &\ge& f(x) + \langle \nabla f(x), y-x \rangle + \omega(\|y-x\|_x), \quad \forall x \in \dom(f), \forall y; \label{f-lowbd} \\ [5pt]
f(y) &\le & f(x) + \langle \nabla f(x), y-x \rangle + \omega_*(\|y-x\|_x) \quad \forall x,y \in \dom(f), \ \|y-x\|_x<1. \label{f-upbd}
\eeqa

\section{Randomized block proximal damped Newton method}
\label{alg}

In this section we propose a randomized block proximal damped Newton (RBPDN)
method for solving problem \eqref{self-concord} in which $g$ is in the form of
\eqref{gx}.

\gap

\noindent
{\bf RBPDN method for solving  \eqref{self-concord}:} \\ [5pt]
Choose $x^0 \in \dom(F)$, $\eta\in[0, 1/4]$,  and $p_i>0$ for $i=1,\ldots,n$ such that $\sum^n_{i=1} p_i=1$. Set $k=0$.
\begin{itemize}
\item[1)] Pick $\i \in \{1,\ldots,n\}$ randomly with probability $p_\i$.
\item[2)]Find an approximate solution $d_\i(x^k)$  to the subproblem
\beq \label{k-subprob}
\min_{d_\i} \left\{f(x^k) + \langle \nabla_\i f(x^k), d_\i\rangle + \frac12 \langle d_\i,
\nabla^2_{\i\i} f(x^k), d_\i \rangle  +g_\i(x^k_\i+d_\i) \right\}
\eeq
such that
\beqa
-v_\i &\in& \nabla_\i f(x^k) + \nabla^2_{\i\i} f(x^k) d_\i(x^k) + \partial g_\i(x^k_\i+d_\i(x^k)),  \label{d-cond1} \\ [6pt]
\|v_\i\|^*_{x^k_\i} &\le& \eta \|d_\i(x^k)\|_{x^k_\i} \label{d-cond2}
\eeqa
for some $v_\i$.
\item[3)] Set  $x^{k+1}_j= x^k_j$ for $j \neq \i$, $x^{k+1}_\i = x^k_\i + d_\i(x^k)/(1+\lambda_\i(x^k))$, $k \leftarrow k+1$ and go to step 1), where
$\lambda_\i(x^k) = \sqrt{\langle d_\i(x^k), \nabla^2_{\i\i}f(x^k) d_\i(x^k)\rangle}$.
\end{itemize}
\noindent
{\bf end}

\gap

{\bf Remark:}
\bi
\item[(i)] The constant $\eta$ controls the inexactness of solving subproblem \eqref{k-subprob}. Clearly, $d_\i(x^k)$ is the optimal solution to
\eqref{k-subprob} if $\eta=0$.
\item[(ii)] For various $g$, the above $d_\i(x^k)$
can be efficiently found. For example, when $g = 0$, $d_\i(x^k)$ can be computed
by conjugate gradient method. In addition, when $g=\|\cdot\|_{\ell_1}$, it can be
found by numerous methods (e.g., see  \cite{BeTe09,Ne07,HaYiZh07,VaFr08,WrNoFi09,JYang,XiZh13,Ulbrich,Byrd,LuCh15}).
\item[(iii)] To verify \eqref{d-cond2}, one has to compute $\|v_\i\|^*_{x^k_\i}$, which can be expensive since $(\nabla^2_{\i\i}f(x^k))^{-1}$ is involved. Alternatively, we may
replace \eqref{d-cond2} by a relation that can be cheaply verified and also ensures \eqref{d-cond2}.  Indeed, as seen later, the sequence $\{x^k\}$ lies in the compact set $\cS(x^0)$ and $\nabla^2 f(x)$
is positive definite for all $x\in\cS(x^0)$. It follows that
\beq \label{sigma-f}
 \sigma_f := \min\limits_{x\in\cS(x^0)} \lambda_{\min}(\nabla^2 f (x))
\eeq
is well-defined and positive,
where $\lambda_{\min}(\cdot)$ denotes the minimal eigenvalue of the associated matrix.
One can observe from \eqref{xi-norm} and \eqref{sigma-f} that
\[
\|v_\i\|^*_{x^k_\i} = \sqrt{v^T_\i (\nabla^2_{\i\i}f(x^k))^{-1} v_\i} \le \frac{\|v_\i\|}{\sqrt{\sigma_f}} .
\]
It follows that if $\|v_\i\| \le \eta \sqrt{\sigma_f} \|d_\i(x^k)\|_{x^k_\i}$ holds,
so does \eqref{d-cond2}. Therefore,  for a cheaper computation, one can replace \eqref{d-cond2} by
\[
\|v_\i\| \le \eta \sqrt{\sigma_f} \|d_\i(x^k)\|_{x^k_\i},
\]
provided that $\sigma_f$ is known or can be bounded from below.
\item[(iv)] The convergence of RBPDN will be analyzed in Section \ref{converge}. In particular, we show that if $g$ is Lipschitz continuous in $\cS(x^0)$, then RBPDN is globally convergent. It is also shown that RBPDN enjoys a local linear convergence. Moreover,
 we show that for a class of $g$ including the case where $g$ is smooth (but not necessarily self-concordant) and $\nabla g$ is Lipschitz continuous in $\cS(x^0)$, RBPDN enjoys a global linear convergence.
\ei


\section{Technical preliminaries}
\label{tech}


In this section we establish some technical results that will be used later to study the
convergence of RBPDN.

For any $x \in \dom(F)$, let $\hd(x)$ be an inexact proximal Newton direction,
which is an approximate solution of
\[
\min\limits_d \left\{ f(x) + \langle \nabla f(x), d\rangle + \frac12 \langle d,\nabla^2 f(x) d\rangle + g(x+d) \right\}
\]
such that
\beq \label{v-subgrad0}
-\hv \in \nabla f(x) + \nabla^2 f(x) \hd(x) + \partial g(x+\hd(x))
\eeq
for some $\hv$ satisfying $\|\hv\|^*_x \le \eta \|\hd(x)\|_x$ with $\eta \in [0,1/4]$.

The following theorem provides an estimate on the reduction of the objective value
resulted from an inexact proximal damped Newton step.

\begin{lemma} \label{thm:descent}
Let $x \in \dom(F)$ and $\hd(x)$ be defined above with $\eta \in [0,1/4]$. Then
\[
F\left(x + \frac{\hd}{1+\hlambda}\right) \le F(x) - \frac12 \omega(\hlambda),
\]
where $\hd=\hd(x)$ and $\hlambda=\|\hd(x)\|_x$.
\end{lemma}

\begin{proof}
By the definition of $\hd$ and $\hlambda$, one can observe that
\[
\|\hd\|_x/(1+\hlambda) = \hlambda/(1+\hlambda) <1.
\]
It then follows from \eqref{f-upbd} that
\beq \label{f-ineq}
f\left(x + \frac{\hd}{1+\hlambda}\right) \le f(x) + \frac{1}{1+\hlambda} \langle \nabla f(x), \hd \rangle +
\omega_*\left(\frac{\hlambda}{1+\hlambda}\right).
\eeq
In view of \eqref{v-subgrad0} and $\hd=\hd(x)$, there exists $s\in \partial g(x+\hd)$ such that
\beq \label{v-subgrad2}
\nabla f(x) + \nabla^2 f(x) \hd + \hv + s = 0.
\eeq
By the convexity of $g$, one has
\beq \label{g-ineq}
g\left(x + \frac{\hd}{1+\hlambda}\right) \le \frac{g(x+\hd) }{1+\hlambda} + \frac{\hlambda  g(x)}{1+\hlambda}
\le \frac{1}{1+\hlambda}[g(x) + \langle s,\hd \rangle] + \frac{\hlambda  g(x)}{1+\hlambda} =
g(x) +  \frac{\langle s,\hd \rangle}{1+\hlambda}.
\eeq
Summing up \eqref{f-ineq} and \eqref{g-ineq}, and using \eqref{v-subgrad2}, we have
\beqa
F\left(x + \frac{\hd}{1+\hlambda}\right) &\le& F(x) + \frac{1}{1+\hlambda} \langle \nabla f(x)+s, \hd \rangle
+ \omega_*\left(\frac{\hlambda}{1+\hlambda}\right) \nn \\ [5pt]
& = & F(x) + \frac{1}{1+\hlambda} \langle -\nabla^2 f(x) \hd-\hv, \hd \rangle
+ \omega_*\left(\frac{\hlambda}{1+\hlambda}\right) \nn \\ [5pt]
&\le& F(x) - \frac{\hlambda^2}{1+\hlambda} + \frac{\hlambda}{1+\hlambda} \|v\|^*_x
+ \omega_*\left(\frac{\hlambda}{1+\hlambda}\right),
\label{F-reduction1}
\eeqa
where the last relation is due to the definition of $\hlambda$ and \eqref{uv-product2}.
In addition, observe from \eqref{omega} that $\omega'(\hlambda) = \hlambda/(1+\hlambda)$. It follows from this and \eqref{conj} that
\[
-\frac{\hlambda^2}{1+\hlambda} + \omega_*\left(\frac{\hlambda}{1+\hlambda}\right) = -\hlambda\omega'(\hlambda)
+ \omega_*\left(\omega'(\hlambda)\right) = -\omega(\hlambda),
\]
which along with \eqref{F-reduction1}, $\|\hv\|^*_x \le \eta \|\hd\|_x$ and $\hlambda=\|\hd\|_x$ implies
\beq \label{F-descent-0}
F\left(x + \frac{\hd}{1+\hlambda}\right) \le F(x) - \omega(\hlambda) + \frac{\eta\hlambda^2}{1+\hlambda}.
\eeq

Claim that for any $\eta \in [0,1/4]$,
\beq \label{claim1}
\frac{\eta\hlambda^2}{1+\hlambda} \le \frac12 \omega(\hlambda).
\eeq
Indeed, let $\phi(\lambda) = \frac12 \omega(\lambda)(1+\lambda)-\eta\lambda^2$. In view of $\omega'(\lambda) = \lambda/(1+\lambda)$, \eqref{omega} and $\eta \in[0,1/4]$, one has that for every $\lambda \ge 0$,
\[
\ba{lcl}
\phi'(\lambda) &=& \frac12 [\omega'(\lambda)(1+\lambda) +\omega(\lambda)] - 2\eta \lambda =  \frac12 \left[\frac{\lambda}{1+\lambda}(1+\lambda) +\lambda-\ln(1+\lambda)\right] - 2\eta \lambda \\ [7pt]
&=& (1-2\eta) \lambda - \frac12\ln(1+\lambda) \ge \frac12 [\lambda - \ln(1+\lambda)] = \frac12 \omega(\lambda) \ge 0.
\ea
\]
This together with $\phi(0)=0$ implies $\phi(\lambda) \ge 0$ . Thus \eqref{claim1} holds as claimed. The conclusion of this lemma then immediately follows from \eqref{F-descent-0} and \eqref{claim1}.
\end{proof}

\vgap

We next provide some lower and upper bounds on the optimality gap.

\begin{lemma} \label{F-gap}
Let $x\in\dom(F)$ and $\blambda(x)$ be defined as
\beq \label{blambda}
\blambda(x) := \min\limits_{s\in\partial F(x)} \|s\|^*_x.
\eeq
Then
\beq \label{opt-gap}
\omega(\|x-x^*\|_{x^*}) \le F(x) - F^* \le \omega_*(\blambda(x)),
\eeq
where the second inequality is valid only when $\blambda(x)<1$.
\end{lemma}

\begin{proof}
Since $x^*$ is the optimal solution of problem \eqref{self-concord}, we have
$-\nabla f(x^*) \in \partial g(x^*)$. This together with the convexity of $g$ implies
$g(x) \ge g(x^*) + \langle -\nabla f(x^*), x-x^* \rangle$.
Also, by \eqref{f-lowbd}, one has
\[
f(x) \ge f(x^*) + \langle \nabla f(x^*), x-x^* \rangle + \omega(\|x-x^*\|_{x^*}).
\]
Summing up these two inequalities yields the first inequality of \eqref{opt-gap}.

Suppose $\blambda(x)<1$. We now prove the second inequality of \eqref{opt-gap}.
Indeed, by  \eqref{f-lowbd}, one has
\[
f(y) \ge f(x) + \langle \nabla f(x), y-x \rangle + \omega(\|y-x\|_x),\quad \forall y.
\]
By \eqref{blambda}, there exists $s\in\partial F(x)$ such that $\|s\|^*_x = \blambda(x)<1$.
Clearly, $s-\nabla f(x) \in \partial g(x)$. In view of this
and the convexity of $g$, we have
\[
g(y) \ge g(x)+\langle s-\nabla f(x), y-x \rangle, \quad \forall y.
\]
Summing up these two inequalities gives
\[
F(y) \ge F(x) + \langle s, y-x \rangle + \omega(\|y-x\|_x), \quad \forall y.
\]
It then follows from this, \eqref{uv-product1} and  \eqref{conj} that
\[
\ba{lcl}
F^* &=& \min\limits_y F(y) \ \ge \ \min\limits_y \left\{F(x) + \langle s, y-x \rangle + \omega(\|y-x\|_x)\right\},
\\ [6pt]
& \ge &  \min\limits_y \left\{F(x) - \|s\|_x^* \cdot \|y-x\|_x + \omega(\|y-x\|_x)\right\},  \\ [10pt]
& \ge & F(x) - \omega_*(\|s\|^*_x) =  F(x) -\omega_*(\blambda(x)),
\ea
\]
where the last inequality uses \eqref{conj}. Thus the second inequality of \eqref{opt-gap} holds.
\end{proof}

\vgap

For the further discussion, we denote by $\td(x)$ and $\tlambda(x)$  the (exact) proximal Newton direction and its local norm at $x\in\dom(F)$, that is,
\beqa
\td(x) &:=& \arg\min\limits_d \left\{f(x) + \langle \nabla f(x), d\rangle + \frac12 \langle d,\nabla^2 f(x) d\rangle
+ g(x+d)\right\}, \label{td} \\
\tlambda(x) &:=& \|\td(x)\|_x. \label{tlambda}
\eeqa

The following result provides an estimate on the reduction of the objective value
resulted from the exact proximal damped Newton step.

\begin{lemma}
Let $x\in\dom(F)$, $\td(x)$  and $\tlambda(x)$ be defined respectively in \eqref{td}
and \eqref{tlambda}, and $\tx = x + \td(x)/(1+\tlambda(x))$. Then
\beqa
F(\tx) \le F(x) - \omega(\tlambda(x)), \label{Newton-reduct} \\ [5pt]
F(x) - F^* \ge  \omega(\tlambda(x)). \label{low-bdd1}
\eeqa
\end{lemma}

\begin{proof}
The relation \eqref{Newton-reduct} follows from \cite[Theorem 5]{TrKyCe15}. In
addition, the relation \eqref{low-bdd1} holds due to  \eqref{Newton-reduct} and  $F(\tx)  \ge F^*$.
\end{proof}

\gap

Throughout the remainder of the paper, let $d_i(x)$ be an approximate solution of
the problem
\beq \label{subprob}
\min_{d_i} \left\{f(x) + \langle \nabla_i f(x), d_i\rangle + \frac12 \langle d_i,
\nabla^2_{ii} f(x), d_i \rangle  +g_i(x_i+d_i) \right\},
\eeq
which satisfies the following conditions:
\beqa
&-v_i \in \nabla_i f(x) + \nabla^2_{ii} f(x) d_i(x) + \partial g_i(x_i+d_i(x)), \label{v-subgrad} \\ [6pt]
&\|v_i\|^*_{x_i} \le \eta \|d_i(x)\|_{x_i} \label{vi}
\eeqa
for some $v_i$ and $\eta\in[0, 1/4]$. Define
\beqa
&&d(x) :=(d_1(x), \ldots, d_n(x)),  \quad\quad v :=(v_1,\ldots,v_n), \label{dx-1} \\
&&\lambda_i(x) := \|d_i(x)\|_{x_i}, \ i=1,\ldots, n, \label{lambdai} \\
&& H(x) := \Diag(\nabla^2_{11}f(x),\ldots,\nabla^2_{n n} f(x)), \label{Hx}
\eeqa
where $H(x)$ is a block diagonal matrix, whose diagonal blocks are  $\nabla^2_{11}f(x), \ldots,  \nabla^2_{n n} f(x)$. It then follows that
\beq \label{v-subgrad-full}
-(\nabla f(x)+v+H(x)d(x)) \in \partial g(x+d(x)).
\eeq

The following result builds some relationship between $\|d(x)\|_{H(x)}$ and  $\sum^n_{i=1} \lambda_i(x)$.

\begin{lemma}
Let $x\in\dom(F)$, $d(x)$, $\lambda_i(x)$ and  $H(x)$ be defined in \eqref{dx-1},  \eqref{lambdai} and \eqref{Hx}, respectively. Then
\beq \label{d-lambda}
\frac{1}{\sqrt{n}} \sum^n_{i=1} \lambda_i(x) \le \|d(x)\|_{H(x)} \le \sum^n_{i=1} \lambda_i(x).
\eeq
\end{lemma}

\begin{proof}
By \eqref{xi-norm}, \eqref{M-norm}, \eqref{dx-1} and \eqref{Hx}, one has
\[
\|d(x)\|_{H(x)} = \sqrt{\sum^n\limits_{i=1} \left\|(\nabla^2_{ii}f(x))^{\frac12} d_i(x)\right\|^2}
\ge \frac{1}{\sqrt{n}} \sum^n\limits_{i=1} \left\|(\nabla^2_{ii}f(x))^{\frac12} d_i(x)\right\|
= \frac{1}{\sqrt{n}}  \sum^n\limits_{i=1} \lambda_i(x),
\]
\[
\sum^n\limits_{i=1} \lambda_i(x) = \sum^n\limits_{i=1}\left \|(\nabla^2_{ii}f(x))^{\frac12} d_i(x)\right\| \ge  \sqrt{\sum^n\limits_{i=1} \left\|(\nabla^2_{ii}f(x))^{\frac12} d_i(x)\right\|^2} = \|d(x)\|_{H(x)}.
\]
\end{proof}

\vgap

The following lemma builds some relationship between $\|d(x)\|_{H(x)}$ and
$\|\td(x)\|_x$.

\begin{lemma}
Let $x\in\dom(F)$, $\td(x)$, $d(x)$ and $H(x)$ be defined in \eqref{td}, \eqref{dx-1} and \eqref{Hx}, respectively. Then
\beqa
\|d(x)\|_{H(x)} &\le & \frac{ \|\td(x)\|_x}{1-\eta} \left((1+\eta)\|H(x)^{\frac12}(\nabla^2 f(x))^{-\frac12}\| +
\|H(x)^{-\frac12}(\nabla^2 f(x))^{\frac12}\|\right), \label{dH} \\
 \|\td(x)\|_x &\le&   \left((1+\eta)\|H(x)^{\frac12}(\nabla^2 f(x))^{-\frac12}\| +
\|H(x)^{-\frac12}(\nabla^2 f(x))^{\frac12}\|\right) \|d(x)\|_{H(x)}. \label{tdH}
\eeqa
\end{lemma}

\begin{proof}
For convenience, let $d=d(x)$, $\td = \td(x)$, $H=H(x)$ and $\tH = \nabla^2 f(x)$. Then it follows from \eqref{v-subgrad-full} and \eqref{td}  that
\[
\ba{l}
-(\nabla f(x)+v+Hd) \in \partial g(x+d), \\ [5pt]
-(\nabla f(x)+\tH \td) \in \partial g(x+\td).
\ea
\]
In view of these and the monotonicity of $\partial g$, one has
$\langle d-\td, -v-Hd+\tH\td\rangle \ge 0$,
 which together with \eqref{M-norm} and \eqref{uv-product2} implies  that
\beqa
 \|d\|^2_H + \|\td\|^2_{\tH}&\le& \langle v, \td-d \rangle + \langle d, (H+\tH) \td \rangle \nn \\ [5pt]
 &\le & \|v\|^*_H (\|d\|_H+\|\td\|_H) + \|d\|_H \cdot \|\td\|_{\tH} \cdot \|H^{-\frac12}(H+\tH)\tH^{-\frac12}\|. \label{lem-ineq1}
\eeqa
Notice that
\beq \label{lem-ineq2}
\|\td\|_H \le \|H^{\frac12}\tH^{-\frac12}\| \cdot\|\td\|_{\tH}.
\eeq
Let  $H_i=\nabla^2_{ii} f(x)$. Observe that $\|v_i\|^*_{H_i}=\|v_i\|^*_{x_i}$ and $\|d_i\|_{H_i}=\|d_i\|_{x_i}$. These
and \eqref{vi} yield $\|v_i\|^*_{H_i} \le \eta \|d_i\|_{H_i}$. In view of this and \eqref{Hx}, one has
\beq
\|v\|^*_H = \sqrt{\sum_i (\|v_i\|^*_{H_i})^2} \le \sqrt{\sum_i \eta^2 \|d_i\|^2_{H_i}} = \eta \|d\|_H. \label{v-norm}
\eeq
It follows from this, \eqref{lem-ineq1} and \eqref{lem-ineq2} that
\beqa
 \|d\|^2_H + \|\td\|^2_{\tH}& \le & \eta \|d\|_H\left (\|d\|_H+\|H^{\frac12}\tH^{-\frac12}\| \cdot\|\td\|_{\tH}\right) + \|d\|_H \cdot \|\td\|_{\tH} \cdot
\|H^{-\frac12}(H+\tH)\tH^{-\frac12}\|,  \nn \\
&\le& \eta \|d\|^2_H + \left((1+\eta) \|H^{\frac12}\tH^{-\frac12}\| +
\|H^{-\frac12}\tH^{\frac12}\|\right) \|d\|_H \cdot\|\td\|_{\tH}, \label{d-td-ineq}
\eeqa
where the second inequality uses the relation
\[
\|H^{-\frac12}(H+\tH)\tH^{-\frac12}\| \le \|H^{\frac12}\tH^{-\frac12}\| + \|H^{-\frac12}\tH^{\frac12}\|.
\]
Clearly, \eqref{d-td-ineq} is equivalent to
\[
(1-\eta) \|d\|^2_H + \|\td\|^2_{\tH} \le  \left((1+\eta) \|H^{\frac12}\tH^{-\frac12}\| +
\|H^{-\frac12}\tH^{\frac12}\|\right) \|d\|_H \cdot\|\td\|_{\tH}.
\]
This, along with $d=d(x)$, $\td = \td(x)$, $H=H(x)$, $\tH = \nabla^2 f(x)$ and
$\|\td\|_x = \|\td\|_{\tH}$, yields \eqref{dH} and \eqref{tdH}.
\end{proof}

\vgap

The following results will be used subsequently to study the convergence of RBPDN.

\begin{lemma}
Let $\cS(x^0)$, $\sigma_f$, $\td(x)$, $d(x)$, $\lambda_i(x)$ and $H(x)$ be defined in
\eqref{S0}, \eqref{sigma-f}, \eqref{td}, \eqref{dx-1}, \eqref{lambdai} and \eqref{Hx},
respectively. Then
\bi
\item[(i)] $\cS(x^0)$ is a nonempty convex compact set.
\item[(ii)]
\beq \label{x-td}
\|x-x^*\|  \le 2 (L_f/\sigma_f) \|\td(x)\|, \quad \forall x\in \cS(x^0),
\eeq
where
\beq  \label{L}
L_f = \max\limits_{x\in\cS(x^0)} \|\nabla^2 f(x)\|.
\eeq
\item[(iii)]
\beq \label{Fgap-lowbd}
F(x) - F^* \ge \omega\left(c_1 \sum^n_{i=1}\lambda_i(x)\right), \quad \forall x\in \cS(x^0),
\eeq
where
\beq \label{c1}
c_1 = \frac{1-\eta}{\sqrt{n}\max\limits_{x\in\cS(x^0)}\left\{(1+\eta)\|H(x)^{\frac12}(\nabla^2 f(x))^{-\frac12}\| +
\|H(x)^{-\frac12}(\nabla^2 f(x))^{\frac12}\|\right\}}.
\eeq
\item[(iv)]
\beq \label{td-d-norm}
\|\td(x)\| \le  \frac{1-\eta}{c_1\sqrt{n\sigma_f}} \|d(x)\|_{H(x)}, \quad \forall x\in \cS(x^0).
\eeq
\item[(v)]
\beq \label{td-norm}
\|\td(x)\| \le \frac{1-\eta}{c_1\sqrt{n\sigma_f}}  \sum^n_{i=1} \lambda_i(x), \quad \forall x\in \cS(x^0).
\eeq
\ei
\end{lemma}

\begin{proof}
(i) Clearly, $\cS(x^0) \neq \emptyset$ due to $x^0 \in \cS(x^0)$.
By \eqref{S0} and the first inequality of \eqref{opt-gap}, one can observe that
$\cS(x^0) \subseteq
 \left\{x: \omega(\|x-x^*\|_{x^*}) \le F(x^0)-F^*\right\}$.
This together with the strict monotonicity of  $\omega$ in $[0,\infty)$ implies that
$\cS(x^0)$ is a bounded set. In addition, we know that $F$ is a closed convex function. Hence, $\cS(x^0)$ is closed and convex.

(ii) By Assumption \ref{assump}, we know that $\nabla^2 f$ is continuous and positive
definite in $\dom(F)$. It follows from this and the compactness of $\cS(x^0)$ that
$\sigma_f$ and $L_f$ are well-defined in \eqref{sigma-f} and \eqref{L} and
moreover they are positive. For convenience, let $\td = \td(x)$ and
$\tH = \nabla^2 f(x)$. By the optimality condition of \eqref{self-concord} and
\eqref{td}, one has
\[
-(\nabla f(x)+\tH \td) \in \partial g(x+\td),  \quad\quad
-\nabla f(x^*) \in \partial g(x^*),
\]
which together with the monotonicity of $\partial g$ yield
\[
\langle x+\td - x^*, -\nabla f(x) -\tH \td + \nabla f(x^*) \rangle \ge 0.
\]
Hence, we have that for all $x\in\cS(x^0)$,
\[
\ba{lcl}
\sigma_f  \|x- x^*\|^2 &\le&  \langle x- x^*, \nabla f(x) - \nabla f(x^*) \rangle
\ \le  \ -\langle \td, \nabla f(x) - \nabla f(x^*) \rangle -  \langle x-x^*, \tH \td  \rangle \\ [6pt]
&\le &  \|\nabla f(x) - \nabla f(x^*)\| \cdot \|\td\|  +   \|\tH\| \cdot \|x-x^*\| \cdot \|\td\| \ \le \  2L_f \|x-x^*\| \cdot \|\td\|,
\ea
\]
which immediately implies \eqref{x-td}.

(iii) In view of \eqref{tlambda}, \eqref{d-lambda}, \eqref{dH} and \eqref{c1}, one can observe that
\[
\tlambda(x) = \|\td(x)\|_x \ge c_1 \sum^n\limits_{i=1} \lambda_i(x), \quad \forall x\in \cS(x^0),
\]
which, together with \eqref{low-bdd1} and the monotonicity of $\omega$ in $[0,\infty)$,  implies that \eqref{Fgap-lowbd} holds.

(iv) One can observe that
\beq \label{td-2norms}
\|\td(x)\| \le \left\|(\nabla^2 f(x))^{-\frac12}\right\| \cdot \|\td(x)\|_x \le \frac{1}{\sqrt{\sigma_f}} \|\td(x)\|_x, \quad \forall x\in \cS(x^0),
\eeq
where the last inequality is due to \eqref{sigma-f}. This, \eqref{tdH} and \eqref{c1} lead to \eqref{td-d-norm}.

(v) The relation \eqref{td-norm} follows from \eqref{d-lambda} and \eqref{td-d-norm}.
\end{proof}

\section{Convergence results}
\label{converge}

In this section we establish some convergence results for RBPDN. In particular, we show
in Subsection \ref{global-converg} that if $g$ is Lipschitz continuous in $\cS(x^0)$, then
RBPDN is globally convergent.  In Subsection \ref{local-linear}, we show that RBPDN enjoys a local linear convergence. In Subsection \ref{global-linear}, we show that for a
class of $g$ including the case where $g$ is smooth (but not necessarily self-concordant)
and $\nabla g$ is Lipschitz continuous in $\cS(x^0)$, RBPDN enjoys a global linear
convergence.

\subsection{Global convergence}
\label{global-converg}

In this subsection we study the global convergence of RBPDN.  To proceed, we first
establish a certain reduction on the objective values over every two consecutive iterations.

\begin{lemma}
Let $\{x^k\}$ be generated by RBPDN. Then
\beq \label{F-reduct}
\bE_{\i}[F(x^{k+1})]  \le  F(x^k) - \frac12 \omega\left(p_{\min}\sum\limits_{i=1}^n\lambda_i(x^k)\right), \quad\quad k \ge 0,
\eeq
where $\lambda_i(\cdot)$ is defined in \eqref{lambdai} and
\beq \label{pmin}
p_{\min} := \min_{1 \le i \le n} p_i.
\eeq
\end{lemma}

\begin{proof}
Recall that $\i\in \{1,\ldots,n\}$ is randomly chosen at iteration $k$ with probability $p_\i$.  Since $f$ is a standard self-concordant function, it is not hard to observe
that $f(x^k_1,\ldots,x^k_{\i-1},z,x^k_{\i+1},\ldots,x^k_n)$ is also a standard
 self-concordant function of $z$.
In view of this and Lemma \ref{thm:descent} with $F$ replaced by $F(x^k_1,\ldots,x^k_{\i-1},z,x^k_{\i+1},\ldots,x^k_n)$, one can obtain that
\beq \label{F-descent}
F(x^{k+1}) \le F(x^k) - \frac12 \omega(\lambda_{\i}(x^k)).
\eeq
Taking expectation with respect to $\i$  and using the convexity of  $\omega$, one has
\[
\ba{lcl}
\bE_{\i}[F(x^{k+1})] &\le& F(x^k) - \frac12 \sum\limits_{i=1}^n p_i\omega(\lambda_i(x^k))
\le F(x^k) - \frac12 \omega\left(\sum\limits_{i=1}^n p_i \lambda_i(x^k)\right) \nn \\ [10pt]
&\le&  F(x^k) - \frac12 \omega\left(p_{\min}\sum\limits_{i=1}^n\lambda_i(x^k)\right),
\ea
\]
where the last inequality follows from \eqref{pmin} and the monotonicity of $\omega$
in $[0,\infty)$.
\end{proof}

\gap

We next show that under a mild assumption RBPDN is globally convergent.

\begin{theorem} \label{global-converg}
Assume that $g$ is Lipschitz continuous in $\cS(x^0)$. Then
\[
\lim\limits_{k \to \infty} \bE[F(x^k)] = F^*.
\]
\end{theorem}

\begin{proof}
It follows from \eqref{F-reduct} that
\[
\ba{lcl}
\bE[F(x^{k+1})] &\le&  \bE[F(x^k)] - \frac12  \bE\left[\omega\left(p_{\min}\sum\limits_{i=1}^n\lambda_i(x^k)\right)\right] \\ [10pt]
&\le & \bE[F(x^k)] - \frac12  \omega\left(p_{\min}\bE\left[\sum\limits_{i=1}^n\lambda_i(x^k)\right]\right),
\ea
\]
where the last relation follows from Jensen's inequality. Hence, we have
\beq \label{expect-ineq}
0 \le \sum_k\omega\left(p_{\min}\bE
\left[\sum\limits_{i=1}^n\lambda_i(x^k)\right]\right) \le F(x^0) - F^*.
\eeq
Notice from \eqref{omega} that $\omega(t)\ge 0$ for all $t \ge 0$ and $\omega(t)=0$
if and only if $t=0$. This and \eqref{expect-ineq} imply that
\beq \label{lim-lambda}
\lim\limits_{k\to\infty} \bE
\left[\sum\limits_{i=1}^n\lambda_i(x^k)\right] = 0.
\eeq
In view of $x^0\in\cS(x^0)$ and \eqref{F-descent}, one can observe that $x^k \in\cS(x^0)$ for all $k \ge 0$. Due to the continuity of $\nabla f$ and the compactness of $\cS(x^0)$,  one can observe that $f$ is Lipschitz continuous in $\cS(x^0)$.
This along with the assumption of Lipschitz continuity of $g$ in $\cS(x^0)$
implies that $F$ is Lipschitz continuous in $\cS(x^0)$ with some Lipschitz constant
$L_F \ge 0$. Using this, \eqref{x-td} and \eqref{td-norm}, we obtain that for all
$k \ge 0$,
\[
\ba{lcl}
 F(x^k) &\le& F^* +L_F \|x^k-x^*\| \le F^*+ \frac{2L_fL_F}{\sigma_f}\|\td(x^k)\|  \\ [6pt]
& \le&  F^* + \frac{2(1-\eta)L_fL_F}{c_1\sqrt{n}\sigma_f^{3/2}}  \sum^n\limits_{i=1} \lambda_i(x^k),
\ea
\]
where the last two inequalities follow from \eqref{x-td} and \eqref{td-norm},
respectively. This together with \eqref{lim-lambda} and $F(x^k) \ge F^*$
implies that the conclusion holds.
\end{proof}

\subsection{Local linear convergence}
\label{local-linear}


In this subsection we show that RBPDN enjoys a local linear convergence.


\begin{theorem}
Let $\{x^k\}$ be generated by RBPDN.
Suppose $F(x^0) \le F^*+\omega(c_1/p_{\min})$, where
$c_1$ and $p_{\min}$ are defined in  \eqref{c1} and \eqref{pmin}, respectively.  Then
\[
\bE[F(x^k)-F^*] \le \left[\frac{6c_2+ p^2_{\min}(1-\theta)}{6c_2+ p^2_{\min}}  \right]^k(F(x^0)-F^*), \quad\quad \forall k \ge 0,
\]
where
\beqa
c_2 &:=& \left |\theta\left[\left(\frac{L_f}{\sigma_f}\right)^{3/2} \frac{2(1-\eta^2)}{c_1\sqrt{n}}-1\right] +
\left(\frac12+\eta\right)  p_{\max}\right|, \label{c4} \\
p_{\max} &:=& \max\limits_{1\le i\le n} p_i, \quad\quad  \theta := \min_{1 \le i \le n} \inf\limits_{x\in\cS(x^0)}\frac{p_i}{1+\lambda_i(x)} \in (0,1),  \label{pmax}
\eeqa
and $\sigma_f$, $L_f$ and  $c_1$ are defined respectively in \eqref{sigma-f}, \eqref{L} and \eqref{c1}.
\end{theorem}

\begin{proof}
Let $k \ge 0$ be arbitrarily chosen. For convenience, let $x=x^k$ and $x^+ = x^{k+1}$.
 By the updating scheme of $x^{k+1}$, one can observe that  $x^+_j = x_j$ for $j \neq \i$ and
\[
x^+_{\i} = x_{\i} + \frac{d_{\i}(x)}{1+\lambda_{\i}(x)},
\]
where $\i\in \{1,\ldots,n\}$ is randomly chosen with probability $p_{\i}$ and
$d_{\i}(x)$ is an approximate solution to problem \eqref{subprob} that satisfies
\eqref{v-subgrad} and \eqref{vi} for some $v_{\i}$ and $\eta\in[0,1/4]$.
To prove this theorem, it suffices to show that
\beq \label{F-loc-linear}
\bE_\i[F(x^+)-F^*] \le \left(\frac{6c_2+ p^2_{\min}(1-\theta)}{6c_2+ p^2_{\min}}  \right)(F(x)-F^*).
\eeq
To this end, we first claim that $\theta$ is well-defined in \eqref{pmax}
and moreover $\theta \in (0,1)$. Indeed, given any $i\in\{1,\ldots,n\}$, let $y \in \Re^N$ be defined as follows:
\[
y_i = x_i + \frac{d_i(x)}{1+\lambda_i(x)}, \quad\quad  y_j=x_j, \  \forall j\neq i,
\]
where $\lambda_i(\cdot)$ is defined in \eqref{lambdai}.
By a similar argument as for \eqref{F-descent}, one has
\[
F(y) \le F(x) - \frac12 \omega(\lambda_i(x)).
\]
Using this, $x\in\cS(x^0)$, $F(y) \ge F^*$ and the monotonicity of
$\omega^{-1}$, we obtain that
\[
\lambda_i(x) \le  \omega^{-1}(2[F(x)-F(y)]) \le \omega^{-1}(2[F(x^0)-F^*]),
\]
where $\omega^{-1}$ is the inverse function of $\omega$ when restricted to the
interval $[0,\infty)$.\footnote{
Observe from \eqref{omega} that $\omega$ is strictly increasing in $[0,\infty)$. Thus,
its inverse function $\omega^{-1}$ is well-defined when restricted to this interval and
moreover it is strictly increasing.} It thus follows that $\theta$ is well-defined in
\eqref{pmax} and moreover $\theta \in (0,1)$.

For convenience, let $\lambda_i = \lambda_i(x)$, $d_i=d_i(x)$ and $H_i=\nabla^2_{ii} f(x)$ for $i=1,\ldots,n$ and $H=\Diag(H_1,\ldots, H_n)$. In view of $x\in\cS(x^0)$ and \eqref{L}, one can observe that
\[
\|H\| \le \|\nabla^2 f(x)\| \le L_f,
\]
which along with \eqref{x-td} and \eqref{td-d-norm} implies
\beqa
\|x-x^*\|_H &\le&  \|H\|^{1/2}  \|x-x^*\| \le 2 (L_f^{3/2}/\sigma_f) \|\td(x)\|, \nn \\
&\le &2 \left(\frac{L_f}{\sigma_f}\right)^{3/2} \frac{1-\eta}{c_1\sqrt{n}} \|d\|_H.
\label{x-xs-norm}
\eeqa
It follows from  \eqref{v-subgrad}
that there exists $s_i \in \partial g_i(x_i+d_i)$ such that
\beq \label{opt-subspace}
 \nabla_i f(x) + H_i d_i+ s_i + v_i = 0, \quad\quad i=1,\ldots, n,
\eeq
which together with the definition of $H$ and $v$ yields
\[
\nabla f(x) + Hd + s + v = 0,
\]
where $s=(s_1,\ldots,s_n) \in \partial g(x+d)$.

By the convexity of $f$, one has
\[
f(x) \le f(x^*) + \langle \nabla f(x), x-x^* \rangle.
\]
In addition, by $s\in \partial g(x+d)$ and the convexity of $g$, one has
\[
g(x+d) \le g(x^*) + \langle s, x+d-x^* \rangle.
\]
Using the last three relations, \eqref{v-norm} and \eqref{x-xs-norm},
 we can obtain that
\beqa
f(x) + \langle \nabla f(x) +v, d \rangle + g(x+d) &\le&
 f(x^*) + \langle \nabla f(x), x-x^* \rangle  + \langle \nabla f(x) +v, d \rangle + g(x^*) \nn \\
&& + \langle s, x+d-x^* \rangle  \nn \\
&= & F^* + \langle \nabla f(x)+v+s, x+d-x^* \rangle - \langle v, x-x^* \rangle  \nn \\
&=& F^* + \langle -Hd, x+d-x^* \rangle - \langle v, x-x^* \rangle  \nn \\
&=& F^* - \langle Hd, d \rangle -\langle Hd, x-x^* \rangle- \langle v, x-x^* \rangle  \nn \\
&\le& F^* -\|d\|^2_H + \|d\|_H \cdot \|x-x^*\|_H + \|v\|^*_H \cdot \|x-x^*\|_H \nn \\
&\le & F^* + \beta \|d\|^2_H, \label{F-ineq1}
\eeqa
where
\beq \label{beta}
\beta = \left(\frac{L_f}{\sigma_f}\right)^{3/2} \frac{2(1-\eta^2)}{c_1\sqrt{n}}-1.
\eeq

By \eqref{vi} and \eqref{pmax}, we have
\beq \label{F-ineq2}
- \sum\limits_i  \frac{p_i\langle v_i, d_i \rangle}{1+\lambda_i} \le
\sum\limits_i  \frac{p_i}{1+\lambda_i} \|v_i\|^*_{H_i} \cdot \|d_i\|_{H_i}
\le \eta \sum\limits_i  \frac{p_i}{1+\lambda_i} \|d_i\|^2_{H_i} \le \eta \ p_{\max}  \|d\|_H^2.
\eeq

In addition, recall that $\omega_*(t)=-t-\ln(1-t)$. It thus follows that
\[
\omega_*(t) = \sum\limits_{k=2}^\infty \frac{t^k}{k!} \ \le \ \frac{t^2}{2}  \sum\limits_{k=0}^\infty t^k\  = \ \frac{t^2}{2(1-t)}, \quad\quad  \forall t\in[0,1).
\]
This inequality implies that
\beq \label{F-ineq3}
\sum\limits_i p_i\omega_*\left(\frac{\lambda_i}{1+\lambda_i}\right) \le
\sum\limits_i \frac{p_i (\lambda_i/(1+\lambda_i))^2}{2(1-\lambda_i/(1+\lambda_i))}
= \frac12 \sum\limits_i \frac{p_i \lambda_i^2}{1+\lambda_i} \le \frac{p_{\max}}{2} \sum\limits_i \lambda_i^2 = \frac{p_{\max}}{2} \|d\|_H^2,
\eeq
where $p_{\max}$ is defined in \eqref{pmax}.

Recall that $s_i\in\partial g_i(x_i+d_i)$. By the convexity of $g_i$, one has $ g_i(x_i+d_i) - g_i(x_i) \le  \langle s_i, d_i \rangle$. It thus follows from this and \eqref{opt-subspace}
that for $i=1,\ldots, n$,
\beqa
& & \langle \nabla_i f(x) + v_i, d_i\rangle  + g_i(x_i+d_i) - g_i(x_i) \le
\langle \nabla_i f(x) + v_i, d_i\rangle + \langle s_i, d_i \rangle \nn \\ [6pt]
&& \ = \langle \nabla_i f(x)  + s_i + v_i , d_i \rangle = - \langle d_i, H_i d_i  \rangle \ \le \ 0.  \label{1st-opt}
\eeqa
By a similar argument as for \eqref{f-ineq} and the definition of $x^+$, one has
\[
f(x^+) \le f(x) +  \frac{1}{1+\lambda_\i}
 \langle \nabla_\i f(x), d_\i \rangle +
\omega_*\left(\frac{\lambda_\i}{1+\lambda_\i}\right).
\]
It also follows from the convexity of $g_\i$ that
\[
g_\i\left(x_\i + \frac{d_\i}{1+\lambda_\i}\right) - g_\i(x_\i) \le \frac{1}{1+\lambda_\i}\left[g_\i(x_\i+d_\i) - g_\i(x_\i)\right].
\]
Using the last two inequalities and the definition of $x^+$, we have
\[
\ba{lcl}
F(x^+) &=&
f(x^+) + g_{\i}\left(x_{\i}+\frac{d_{\i}}{1+\lambda_{\i}}\right) + \sum\limits_{j\neq \i} g_j(x_j) \\ [6pt]
&=& f(x^+) + g(x) +  g_{\i}\left(x_{\i}+\frac{d_{\i}}{1+\lambda_{\i}}\right) -g_\i(x_\i) \\ [6pt]
&\le & f(x) +  \frac{1}{1+\lambda_\i}
 \langle \nabla_\i f(x), d_\i \rangle +
\omega_*\left(\frac{\lambda_\i}{1+\lambda_\i}\right) + g(x) +  g_{\i}\left(x_{\i}+\frac{d_{\i}}{1+\lambda_{\i}}\right) -g_\i(x_\i) \\ [6pt]
&=& F(x) +  \frac{1}{1+\lambda_\i}
 \langle \nabla_\i f(x), d_\i \rangle +
\omega_*\left(\frac{\lambda_\i}{1+\lambda_\i}\right) +  g_{\i}\left(x_{\i}+\frac{d_{\i}}{1+\lambda_{\i}}\right) -g_\i(x_\i) \\ [6pt]
&\le&  F(x) +  \frac{1}{1+\lambda_\i}
 \langle \nabla_\i f(x), d_\i \rangle +
\omega_*\left(\frac{\lambda_\i}{1+\lambda_\i}\right)+ \frac{1}{1+\lambda_\i}\left[g_\i(x_\i+d_\i) - g_\i(x_\i)\right] \\ [6pt]
&=& F(x) +  \frac{1}{1+\lambda_\i}
\left[\langle \nabla_\i f(x) +v_\i, d_\i \rangle + g_\i(x_\i+d_\i) - g_\i(x_\i)\right]
-  \frac{\langle v_\i, d_\i \rangle}{1+\lambda_\i} +
\omega_*\left(\frac{\lambda_\i}{1+\lambda_\i}\right).
\ea
\]
Taking expectation with respect to $\i$ on both sides and using \eqref{pmax},
 \eqref{F-ineq1}, \eqref{F-ineq2}, \eqref{F-ineq3} and \eqref{1st-opt}, one has
{\small
\beqa
\bE_\i[F(x^+)] &\le& F(x) + \sum\limits_i \frac{p_i}{1+\lambda_i}
\underbrace{\left[\langle \nabla_i f(x) +v_i, d_i \rangle + g_i(x_i+d_i) - g_i(x_i)\right]}_{\le 0\  \mbox{due  to} \ \eqref{1st-opt}}
 -   \sum\limits_i  \frac{p_i\langle v_i, d_i \rangle}{1+\lambda_i} + \sum\limits_i
p_i \omega_*\left(\frac{\lambda_i}{1+\lambda_i}\right) \nn \\
&\le&  F(x) + \theta \sum\limits_i \left[\langle \nabla_i f(x) +v_i, d_i \rangle + g_i(x_i+d_i) - g_i(x_i)\right]  -   \sum\limits_i  \frac{p_i\langle v_i, d_i \rangle}{1+\lambda_i} + \sum\limits_i
p_i \omega_*\left(\frac{\lambda_i}{1+\lambda_i}\right)  \nn \\
&=&  F(x) +  \theta \left[\langle \nabla f(x) +v, d \rangle + g(x+d) - g(x)\right]   -   \sum\limits_i  \frac{p_i\langle v_i, d_i \rangle}{1+\lambda_i} + \sum\limits_i
p_i \omega_*\left(\frac{\lambda_i}{1+\lambda_i}\right) \nn \\
&=&  (1-\theta)F(x) +  \theta \left[f(x)+\langle \nabla f(x) +v, d \rangle + g(x+d) \right]   -   \sum\limits_i  \frac{p_i\langle v_i, d_i \rangle}{1+\lambda_i} + \sum\limits_i
p_i \omega_*\left(\frac{\lambda_i}{1+\lambda_i}\right)  \nn \\
&\le&  (1-\theta)F(x) +\theta (F^* +\beta \|d\|^2_H) +\eta \ p_{\max}  \|d\|_H^2 +
\frac{p_{\max}}{2} \|d\|_H^2 \nn  \\
&=&  (1-\theta)F(x) +\theta F^* + \left(\theta\beta +
(1/2+\eta)  p_{\max}\right)  \|d\|_H^2 \nn \\
&\le&  (1-\theta)F(x) +\theta F^* +  c_2\left (\sum_i \lambda_i\right)^2, \label{F-ineq4}
\eeqa
}
where the last inequality is due to \eqref{beta}, \eqref{c4} and $\|d\|_H^2=\sum_i \lambda^2_i \le (\sum_i \lambda_i)^2$.

One can easily observe from \eqref{F-ineq4} that the conclusion of this theorem holds if $c_2=0$. We now assume $c_2>0$. Let $\delta^+=F(x^+)-F^*$ and $\delta=F(x)-F^*$. It then follows from \eqref{F-ineq4} that
\[
\bE_\i[\delta^+] \le (1-\theta) \delta + c_2 \left(\sum_i \lambda_i\right)^2,
\]
which yields
\beq \label{lambda-lbd}
\left(\sum_i \lambda_i\right)^2 \ge \frac{1}{c_2} \left(\bE_\i[\delta^+] - (1-\theta) \delta\right)
\eeq
By the assumption, one has $F(x) \le F(x^0) \le F^*+\omega(c_1/p_{\min})$. By this and \eqref{Fgap-lowbd},  we have
\[
\omega(c_1 \sum_i \lambda_i) \le  F(x) - F^* \le \omega(c_1/p_{\min}),
\]
which together with the monotonicity of $\omega$ in $[0,\infty)$ implies
 $p_{\min}\sum_i \lambda_i \le 1 $. Observe that
\[
\omega(t) = t-\ln(1+t) = \sum^\infty_{k=2}\frac{(-1)^k t^k}{k!} \ge \frac{t^2}{2}-\frac{t^3}{6} \ge \frac{t^2}{3}, \quad\quad \forall t\in [0,1].
\]
This and $p_{\min}\sum_i \lambda_i \le 1$ lead to
\[
\omega\left(p_{\min}\sum_i \lambda_i\right) \ge \frac13 p^2_{\min}\left(\sum_i \lambda_i\right)^2.
\]
It then follows from this and \eqref{F-reduct} that
\[
\bE_\i[\delta^+] \le \delta - \frac16 p^2_{\min}\left(\sum_i \lambda_i\right)^2,
\]
which together with \eqref{lambda-lbd} gives
\[
\bE_\i[\delta^+] \le \delta - \frac{p^2_{\min}}{6c_2} \left(\bE_\i[\delta^+] - (1-\theta) \delta\right).
\]
Hence, we obtain that
\[
\bE_\i[\delta^+] \le \left(\frac{6c_2+ p^2_{\min}(1-\theta)}{6c_2+ p^2_{\min}}  \right)\delta,
\]
which proves \eqref{F-loc-linear} as desired.
\end{proof}

\subsection{Global linear convergence}
\label{global-linear}

In this subsection we show that for a class of $g$ including the case where $g$ is smooth
(but not necessarily self-concordant) and $\nabla g$ is Lipschitz continuous in
$\cS(x^0)$,\footnote{This covers the case where  $g = 0$, which, for instance, arises in the interior point methods for solving smooth convex optimization problems.} RBPDN enjoys a global linear convergence. To this end,  we make the
following assumption throughout this subsection which, as shown subsequently, holds
for a class of $g$.

\begin{assumption} \label{assump2}
There exists some $c_3>0$ such that
\[
\|\td(x)\| \ge c_3 \blambda(x), \quad\quad \forall x\in\cS(x^0),
\]
where $\cS(x^0)$, $\blambda(x)$ and $\td(x)$ are defined in \eqref{S0}, \eqref{blambda} and \eqref{td}, respectively.
\end{assumption}

The following proposition shows that  Assumption \ref{assump2} holds for a class of $g$
including $g=0$ as a special case.

\begin{proposition} \label{suff-cond-assump2}
Suppose that $g$ is Lipschitz differentiable in $\cS(x^0)$ with a Lipschitz constant
$L_g \ge 0$. Then Assumption \ref{assump2} holds with $c_3=\sqrt{\sigma_f}/(L_f+L_g)$, where
 $\sigma_f$ and $L_f$ are defined in \eqref{sigma-f} and \eqref{L}, respectively.
\end{proposition}

\begin{proof}
Let $x \in \cS(x^0)$ be arbitrarily chosen. It follows from \eqref{td} and the differentiability of $g$ that
\[
\nabla f(x) +\nabla^2 f(x) \td(x) + \nabla g(x+\td(x)) = 0,
\]
which, together with  \eqref{blambda}, \eqref{L} and the Lipschitz continuity of
$\nabla g$,  implies that
\[
\ba{lcl}
\blambda(x)  &=& \|\nabla f(x) + \nabla g(x)\|^*_x  \le \frac{1}{\sqrt{\sigma_f}} \|\nabla f(x) + \nabla g(x)\|, \\ [8pt]
&=&
  \frac{1}{\sqrt{\sigma_f}} \|\nabla g(x) - \nabla g(x+\td(x))  -  \nabla^2 f(x) \td(x)\| \le  \frac{L_f + L_g}{\sqrt{\sigma_f}} \|\td(x)\|.
\ea
\]
and hence the conclusion holds.
\end{proof}

\vgap

We next provide a lower bound for $\blambda(x)$ in terms of the optimality gap, which
will play crucial role in our subsequent analysis.

\begin{lemma}
Let $x\in\dom(F)$ and $\blambda(x)$ be defined in \eqref{blambda}.  Then
\beq \label{blambda-lowbdd}
\blambda(x) \ge \omega_*^{-1}(F(x)-F^*),
\eeq
where $\omega_*^{-1}$ is the inverse function of $\omega_*$ when restricted to the interval $[0,1)$.
\end{lemma}

\begin{proof}
Observe from \eqref{omega} that $\omega_*(t)\in [0, \infty)$ for $t\in[0,1)$ and
$\omega_*$ is strictly increasing in $[0,1)$. Thus its inverse function $\omega_*^{-1}$
is well-defined when restricted to this interval. It also follows that $\omega_*^{-1}(t) \in [0,1)$ for $t\in[0,\infty)$ and $\omega_*^{-1}$ is strictly increasing in $[0,\infty)$.
We divide the rest of the proof into two separable cases as follows.

Case 1): $\blambda(x)<1$. It follows from Theorem \ref{F-gap} that
$F(x)-F^* \le \omega_*(\blambda(x))$. Taking $\omega_*^{-1}$ on both sides of
this relation and using the monotonicity of $\omega_*^{-1}$, we see that \eqref{blambda-lowbdd} holds.

Case 2): $\blambda(x) \ge 1$. \eqref{blambda-lowbdd} clearly holds in this case
due to $\omega_*^{-1}(t) \in [0,1)$ for all $t \ge 0$
\end{proof}

\gap

In what follows, we show that under Assumption \ref{assump2} RBPDN enjoys a
global linear convergence.

\begin{theorem} \label{thm:global-linear-converg}
Let $\{x^k\}$ be generated by RBPDN. Suppose that Assumption \ref{assump2} holds. Then
\[
\bE[F(x^k)-F^*] \le  \left[1-\frac{c^2_4 p^2_{\min}(1- \omega_*^{-1}(\delta_0))} {2(1+c_4 p_{\min} \omega_*^{-1}(\delta_0))}\right]^k(F(x^0)-F^*), \quad\quad \forall k \ge0,
\]
where $\delta_0 = F(x^0)-F^*$,
\beq \label{c3}
c_4 =  \frac{c_1c_3\sqrt{n\sigma_f}}{1-\eta},
\eeq
and $\sigma_f$ and $c_1$ are defined in \eqref{sigma-f} and \eqref{c1}, respectively.
\end{theorem}

\begin{proof}
Let $k \ge 0$ be arbitrarily chosen. For convenience, let $x=x^k$ and $x^+ = x^{k+1}$.
 By the updating scheme of $x^{k+1}$, one can observe that  $x^+_j = x_j$ for $j \neq \i$ and
\[
x^+_{\i} = x_{\i} + \frac{d_{\i}(x)}{1+\lambda_{\i}(x)},
\]
where $\i\in \{1,\ldots,n\}$ is randomly chosen with probability $p_{\i}$ and
$d_{\i}(x)$ is an approximate solution to problem \eqref{subprob} that satisfies
\eqref{v-subgrad} and \eqref{vi} for some $v_{\i}$ and $\eta\in[0,1/4]$.
To prove this theorem, it suffices to show that
\beq \label{F-linear}
\bE_{\i}[F(x^+)-F^*] \le  \left[1-\frac{c^2_4 p^2_{\min}(1- \omega_*^{-1}(\delta_0))} {2(1+c_4 p_{\min} \omega_*^{-1}(\delta_0))}\right](F(x)-F^*).
\eeq

Indeed, it follows from \eqref{td-norm}, \eqref{c3}
and Assumption \ref{assump2} that
\[
 \sum^n\limits_{i=1} \lambda_i(x) \ \ge \  \frac{c_1\sqrt{n\sigma_f}}{1-\eta} \|\td(x)\|
\ \ge \  c_4 \blambda(x).
\]
This together with \eqref{blambda-lowbdd} yields
\[
 \sum^n\limits_{i=1} \lambda_i(x) \ \ge \ c_4 \omega_*^{-1}(F(x)-F^*).
\]
Using this, \eqref{F-reduct} and the monotonicity of $\omega$ in $[0,\infty)$, we obtain that
\[
\bE_{\i}[F(x^+)] \le F(x) - \frac12 \omega\left(c_4 p_{\min} \omega_*^{-1}(F(x)-F^*)\right).
\]
Let $\delta^+=F(x^+)-F^*$ and $\delta=F(x)-F^*$. It then follows that
\beq \label{expect-reduct}
\bE_{\i}[\delta^+] \le \delta - \frac12 \omega\left(c_4 p_{\min} \omega_*^{-1}(\delta)\right).
\eeq

Consider the function $t= \omega_*^{-1}(s)$. Then $s=\omega_*(t)$. Differentiating both sides with
respect to $s$, we have
\[
(\omega_*(t))' \frac{dt}{ds} = 1,
\]
which along with $\omega_*(t)=-t-\ln(1-t)$ yields
\[
(\omega_*^{-1}(s))' = \frac{dt}{ds} = \frac{1}{(\omega_*(t))'} = \frac{1-t}{t} =
\frac{1- \omega_*^{-1}(s)}{ \omega_*^{-1}(s)}.
\]
In view of this and $\omega(t)=t-\ln(1+t)$, one has that for any $\alpha>0$,
\beq \label{deriv-omega}
\frac{d}{ds} [\omega(\alpha\omega_*^{-1}(s))] = \alpha \omega'(\alpha \omega_*^{-1}(s))(\omega_*^{-1}(s))' =
\alpha \cdot \frac{\alpha \omega_*^{-1}(s)}{1+\alpha \omega_*^{-1}(s)} \cdot
\frac{1- \omega_*^{-1}(s)}{ \omega_*^{-1}(s)} = \frac{\alpha^2(1- \omega_*^{-1}(s))}{1+\alpha \omega_*^{-1}(s)}.
\eeq
Notice that $\delta \le \delta_0$ due to $x\in\cS(x^0)$.  By this and the monotonicity of $\omega_*^{-1}$, one can see that
\[
\omega_*^{-1}(s) \le \omega_*^{-1}(\delta) \le \omega_*^{-1}(\delta_0), \quad\quad
\forall s \in [0,\delta],
\]
which implies that
\[
 \frac{1- \omega_*^{-1}(s)}{1+\alpha \omega_*^{-1}(s)} \ge  \frac{1- \omega_*^{-1}(\delta_0)}{1+\alpha \omega_*^{-1}(\delta_0)}, \quad\quad \forall s \in [0,\delta].
\]
Also, observe that $\omega(\alpha\omega_*^{-1}(0))=0$. Using these relations and \eqref{deriv-omega}, we have
 \[
\omega(\alpha\omega_*^{-1}(\delta)) = \int^\delta_0 \frac{d}{ds} [\omega(\alpha\omega_*^{-1}(s))] ds = \int^\delta_0  \frac{\alpha^2(1- \omega_*^{-1}(s))}{1+\alpha \omega_*^{-1}(s)} ds
\ge    \frac{\alpha^2(1- \omega_*^{-1}(\delta_0))}{1+\alpha \omega_*^{-1}(\delta_0)}\delta.
\]
This and \eqref{expect-reduct} with $\alpha=c_4p_{\min}$ lead to
\[
\bE_{\i}[\delta^+] \le  \left[1-\frac{c^2_4 p^2_{\min}(1- \omega_*^{-1}(\delta_0))} {2(1+c_4 p_{\min} \omega_*^{-1}(\delta_0))}\right]\delta,
\]
which gives \eqref{F-linear} as desired.
 \end{proof}

\gap

The following result is an immediate consequence of Proposition \ref{suff-cond-assump2}
and Theorem \ref{thm:global-linear-converg}.

\begin{corollary} \label{cor:global-linear}
Let $\{x^k\}$ be generated by RBPDN. Suppose that $g$ is Lipschitz
differentiable in $\cS(x^0)$ with a Lipschitz constant $L_g \ge 0$. Then
\[
\bE[F(x^k)-F^*] \le  \left[1-\frac{\tc^2_4 p^2_{\min}(1- \omega_*^{-1}(\delta_0))} {2(1+\tc_4 p_{\min} \omega_*^{-1}(\delta_0))}\right]^k(F(x^0)-F^*), \quad\quad \forall k \ge0,
\]
where $\delta_0 = F(x^0)-F^*$,
\[
\tc_4 =  \frac{ \sqrt{n}c_1\sigma_f}{(1-\eta)(L_f+L_g)},
\]
and $\sigma_f$, $L_f$ and  $c_1$ are defined in \eqref{sigma-f}, \eqref{L} and \eqref{c1}, respectively.
\end{corollary}

One can observe that RBPDN reduces to PDN \cite{TrKyCe15} or DN \cite{Ne04}  \footnote{PDN becomes DN if $g=0$.} by setting $n=1$. It thus follows from Corollary \ref{cor:global-linear} that PDN for a class of $g$ and DN are globally linearly convergent, which is stated below. To the best of our knowledge, this result was previously unknown in the literature.

\begin{corollary} \label{cor:global-linear-1}
Suppose that $g$ is Lipschitz differentiable in $\cS(x^0)$. Then PDN \cite{TrKyCe15} for such $g$ and DN \cite{Ne04} are globally linearly convergent.
\end{corollary}

Before ending this subsection we show that Corollary \ref{cor:global-linear-1} can be
used to sharpen the existing iteration complexity of some methods in
\cite{Ne04,ZhXi15-1,TrKyCe15}.

A mixture of DN and Newton methods is presented in \cite[Section 4.1.5]{Ne04} for solving problem \eqref{self-concord} with $g=0$.  In particular, this method consists of two stages. Given an initial point $x^0$, $\beta \in (0,(3-\sqrt{5})/2)$ and $\epsilon>0$,
 the first stage performs the DN iterations
\beq \label{PDN-iter}
x^{k+1} = x^k - \frac{\td(x^k)}{1+\tlambda(x^k)}
\eeq
until finding some $x^{K_1}$ such that  $\tlambda(x^{K_1}) \le \beta$,
where $\td(\cdot)$ and $\tlambda(\cdot)$ are defined in \eqref{td} and
\eqref{tlambda}, respectively.  The second stage executes the standard Newton iterations
\beq \label{NT}
x^{k+1} = x^k - \td(x^k),
\eeq
starting at $x^{K_1}$ and  terminating at some $x^{K_2}$ such that
$\tlambda(x^{K_2}) \le \epsilon$. As shown in \cite[Section 4.1.5]{Ne04},
the second stage converges quadratically:
\beq \label{local-q}
\tlambda(x^{k+1}) \le \left(\frac{\tlambda(x^k)}{1-\tlambda(x^k)}\right)^2, \quad\quad\forall k \ge K_1.
\eeq
In addition, an upper bound on $K_1$ is established in \cite[Section 4.1.5]{Ne04},
which is
\beq \label{bound-K1}
K_1 \le \left\lceil(F(x^0)-F^*)/\omega(\beta)\right\rceil.
\eeq
In view of \eqref{local-q}, one can easily show that
\beq \label{K2-K1}
K_2 - K_1 \le \left\lceil \log_2\left(\frac{\log \epsilon -2\log(1-\beta)}{\log\beta -2\log(1-\beta)}\right) \right\rceil.
\eeq
Observe that the first stage of this method is just DN, which is a special case of RBPDN
with $n=1$ and $\eta=0$. It thus follows from Corollary \ref{cor:global-linear-1} that
 the first stage converges linearly. In fact, it can be shown that
\beq \label{1st-stage-gap}
F(x^{k+1}) - F^* \le \left(1-\frac{1-\omega_*^{-1}(\delta_0)}{1+\omega_*^{-1}(\delta_0)}\right) (F(x^k) - F^*), \quad\quad \forall k \le K_1,
\eeq
where $\delta_0=F(x^0)-F^*$. Indeed,  since $g=0$, one can observe from
\eqref{blambda} and \eqref{tlambda} that $\tlambda(x^k) = \blambda(x^k)$. It then
follows from this, $g=0$ and \cite[Theorem 4.1.12]{Ne04} that $F(x^{k+1}) \le F(x^k) - \omega(\blambda(x^k))$ for all $k \le K_1$. This together with \eqref{blambda-lowbdd} implies that
\[
F(x^{k+1}) \le F(x^k) - \omega(\omega_*^{-1}(F(x^k)-F^*)), \quad\quad\forall k \le K_1.
\]
The relation \eqref{1st-stage-gap} then follows from this and a similar argument as
in the proof of Theorem \ref{thm:global-linear-converg}. Let
\[
\bK =  \left\lceil \left[\frac{\log(\omega(\beta))-\log\delta_0}{\log\left(1-\frac{1-\omega_*^{-1}(\delta_0)}{1+\omega_*^{-1}(\delta_0)}\right)}\right]_+\right\rceil,
\]
where $t_+ = \max(t,0)$. In view of \eqref{1st-stage-gap}, one can easily verify that $F(x^{\bK})-F^* \le \omega(\beta)$,
which  along with \eqref{low-bdd1} implies that  $\tlambda(x^{\bK}) \le \beta$.
By \eqref{bound-K1} and the definition of $K_1$, one can have
$K_1 \le \min\left\{\bK, \left\lceil \delta_0/\omega(\beta)\right\rceil \right\}$,
which sharpens the bound \eqref{bound-K1}. Combining this relation and \eqref{K2-K1},
we thus obtain the following new iteration complexity for finding an approximate solution
of \eqref{self-concord} with $g=0$ by a mixture of DN and Newton method \cite[Section 4.1.5]{Ne04}.

\begin{theorem}
Let $x^0 \in \dom(F)$, $\beta \in (0,(3-\sqrt{5})/2)$ and $\epsilon>0$ be given.
Then the mixture of DN and Newton methods \cite[Section 4.1.5]{Ne04} for solving problem \eqref{self-concord} with $g=0$ requires at most
\[
\min\left\{
 \left\lceil \left[\frac{\log(\omega(\beta))-\log\delta_0}{\log\left(1-\frac{1-\omega_*^{-1}(\delta_0)}{1+\omega_*^{-1}(\delta_0)}\right)}\right]_+\right\rceil, \left\lceil \frac{\delta_0}{\omega(\beta)}\right\rceil  \right\} +  \left\lceil \log_2\left(\frac{\log \epsilon -2\log(1-\beta)}{\log\beta -2\log(1-\beta)}\right) \right\rceil
\]
iterations for finding some $x^k$ satisfying $\tlambda(x^k) \le \epsilon$, where $\delta_0=F(x^0)-F^*$.
\end{theorem}

Recently, Zhang and Xiao \cite{ZhXi15-1} proposed an inexact DN method for
solving problem \eqref{self-concord} with $g=0$, whose iterations are updated as follows:
\[
x^{k+1} = x^k - \frac{\hd(x^k)}{1+\hlambda(x^k)}, \quad\quad \forall k \ge 0,
\]
where $\hd(x^k)$ is an approximation to $\td(x^k)$ and
$\hlambda(x^k)=\sqrt{\langle \hd(x^k), \nabla^2 f(x^k) \hd(x^k)\rangle}$ (see
\cite[Algorithm 1]{ZhXi15-1} for details). It is shown in \cite[Theorem 1]{ZhXi15-1} that
such $\{x^k\}$ satisfies
\beqa
F(x^{k+1}) \le F(x^k) - \frac12\omega(\tlambda(x^k)),& & \forall k \ge 0, \label{ZX-1} \\ [6pt]
\omega(\tlambda(x^{k+1})) \le \frac12 \omega(\tlambda(x^k)), & & \mbox{if} \ \tlambda(x^k) \le 1/6, \label{ZX-2}
\eeqa
where $\tlambda(\cdot)$ is defined in \eqref{tlambda}. These relations are used in
\cite{ZhXi15-1} for deriving an iteration complexity of the inexact DN method. In
particular,  its complexity analysis is divided into two parts. The first part estimates
the number of iterations required for generating some $x^{K_1}$ satisfying
$\tlambda(x^{K_1}) \le 1/6$, while the second part estimates the additional
iterations needed for generating some $x^{K_2}$ satisfying $F(x^{K_2})-F^* \le \epsilon$. In \cite{ZhXi15-1}, the relation \eqref{ZX-1} is used to show that
\beq \label{ZX-K1}
K_1 \le \left\lceil (2(F(x^0)-F^*))/\omega(1/6)\right\rceil,
\eeq
while \eqref{ZX-2} is used to establish
\beq \label{ZX-K2-K1}
K_2 - K_1 \le \left\lceil \log_2\left(\frac{2\omega(1/6)}{\epsilon}\right)\right\rceil.
\eeq
It follows from these two relations that the inexact DN method can find an approximate
solution $x^k$ satisfying $F(x^k)-F^* \le \epsilon$ in at most
\[
\left\lceil \frac{2(F(x^0)-F^*)}{\omega(1/6)}\right\rceil  +  \left\lceil \log_2\left(\frac{2\omega(1/6)}{\epsilon}\right)\right\rceil
\]
iterations, which is stated in \cite[Corollary 1]{ZhXi15-1}.

By a similar analysis as above, one can show that the inexact DN method (\cite[Algorithm 1]{ZhXi15-1}) is globally linearly convergent. In fact,  it can be shown that
\beq \label{ZhXi-linear}
F(x^{k+1}) - F^* \le \left(1-\frac{1-\omega_*^{-1}(\delta_0)}{2(1+\omega_*^{-1}(\delta_0))}\right) (F(x^k) - F^*), \quad\quad \forall k \ge 0,
\eeq
where $\delta_0=F(x^0)-F^*$. Indeed,  since $g=0$, one has $\tlambda(x^k) = \blambda(x^k)$. It  follows from this, \eqref{blambda-lowbdd}
and \eqref{ZX-1} that
\[
F(x^{k+1}) \le F(x^k) - \frac12\omega(\omega_*^{-1}(F(x^k)-F^*)), \quad\quad\forall k \ge 0.
\]
The relation \eqref{ZhXi-linear} then follows from this and a similar derivation as
in the proof of Theorem \ref{thm:global-linear-converg}.  By \eqref{ZX-K1}, \eqref{ZhXi-linear} and a similar argument as above,  one can have
\[
K_1 \le \min\left\{
\left\lceil \left[\frac{\log(\frac12\omega(1/6))- \log\delta_0}{\log\left(1-\frac{1-\omega_*^{-1}(\delta_0)}{2(1+\omega_*^{-1}(\delta_0))}\right)}\right]_+\right\rceil, \left\lceil \frac{2\delta_0}{\omega(1/6)}\right\rceil  \right\},
\]
which improves the bound \eqref{ZX-K1}. Combining this relation and \eqref{ZX-K2-K1},
we thus obtain the following new iteration complexity for finding an approximate solution
of  \eqref{self-concord} with $g=0$ by the aforementioned inexact DN method.

\begin{theorem}
Let $x^0 \in \dom(F)$ and $\epsilon>0$ be given. Then the inexact DN method (\cite[Algorithm 1]{ZhXi15-1}) for solving problem \eqref{self-concord} with $g=0$ requires at most
\[
\min\left\{
\left\lceil \left[\frac{\log(\frac12\omega(1/6))- \log\delta_0}{\log\left(1-\frac{1-\omega_*^{-1}(\delta_0)}{2(1+\omega_*^{-1}(\delta_0))}\right)}\right]_+\right\rceil, \left\lceil \frac{2\delta_0}{\omega(1/6)}\right\rceil  \right\} +  \left\lceil \log_2\left(\frac{2\omega(1/6)}{\epsilon}\right)\right\rceil
\]
iterations for finding some $x^k$ satisfying $F(x^k)-F^* \le \epsilon$, where $\delta_0=F(x^0)-F^*$.
\end{theorem}

Dinh-Tran et al.\ recently proposed in \cite[Algorithm 1]{TrKyCe15} a proximal Newton
method for solving problem \eqref{self-concord} with general $g$. Akin to the
aforementioned method \cite[Section 4.1.5]{Ne04} for \eqref{self-concord} with $g=0$,
this method  also consists of two stages (or phases). The first stage performs the PDN
iterations in the form of \eqref{PDN-iter} for finding some $x^{K_1}$ such that
$\tlambda(x^{K_1}) \le \omega(0.2)$, while the second stage executes the proximal
Newton iterations in the form of \eqref{NT} starting at $x^{K_1}$ and terminating at
some $x^{K_2}$ such that $\tlambda(x^{K_2}) \le \epsilon$. As shown in \cite[Theorem 6]{TrKyCe15}, the second stage converges quadratically. The following relations are  essentially established in  \cite[Theorem 7]{TrKyCe15}:
\beqa
K_1 &\le& \left\lceil(F(x^0)-F^*)/\omega(0.2)\right\rceil, \label{bound-K1-Tr} \\
K_2 - K_1 &\le&\left\lceil 1.5\log\log\frac{0.28}{\epsilon}\right\rceil.   \label{K2-K1-Tr}
\eeqa
Throughout  the remainder of this subsection, suppose that Assumption \ref{assump2}
holds. Observe that the first stage of this method is just PDN, which is a special case of
RBPDN with $n=1$ and $\eta=0$. It thus follows from Corollary \ref{cor:global-linear-1}
that the first stage converges linearly. In fact, it can be shown that
\beq \label{Dinh-1st-stage}
F(x^{k+1}) - F^* \le  \left[1-\frac{\hc^2  (1- \omega_*^{-1}(\delta_0))} {(1+\hc  \omega_*^{-1}(\delta_0)}\right]^k(F(x^0)-F^*),  \quad\quad \forall k \le K_1,
\eeq
where $\delta_0 = F(x^0)-F^*$, $\hc  =  c_3\sqrt{\sigma_f}$, and $\sigma_f$ and $c_3$ are given in \eqref{sigma-f} and Assumption \ref{assump2}, respectively.
Indeed, by \eqref{tlambda} and \eqref{td-2norms}, one has $\|\td(x^k)\| \le \tlambda(x^k)/\sqrt{\sigma_f}$.
In addition, by Assumption \ref{assump2}, we have $\|\td(x^k)\| \ge c_3 \blambda(x^k)$. It follows from these two relations that
$\tlambda(x^k) \ge  \hc \blambda(x^k)$, which together with
\eqref{blambda-lowbdd} yields
 $\tlambda(x^k) \ge \hc \omega_*^{-1}(F(x^k)-F^*)$.
This and \eqref{Newton-reduct} imply that
\[
F(x^{k+1}) \le F(x^k) - \omega(\hc\omega_*^{-1}(F(x^k)-F^*)), \quad\quad\forall k \le K_1.
\]
The relation \eqref{Dinh-1st-stage} then follows from this and a similar argument as
in the proof of Theorem \ref{thm:global-linear-converg}. Let
\[
\bK =  \left\lceil \left[\frac{\log(\omega(0.2))-\log\delta_0}{\log\left(1-\frac{\hc^2  (1- \omega_*^{-1}(\delta_0))} {(1+\hc  \omega_*^{-1}(\delta_0)}\right)}\right]_+\right\rceil.
\]
By \eqref{Dinh-1st-stage}, one can easily verify that $F(x^{\bK})-F^* \le \omega(0.2)$,
which along with \eqref{low-bdd1} implies that  $\tlambda(x^{\bK}) \le 0.2$.
By \eqref{bound-K1} and the definition of $K_1$, one can have
$K_1 \le \min\left\{\bK, \left\lceil \delta_0/\omega(0.2)\right\rceil \right\}$,
which sharpens the bound \eqref{bound-K1-Tr}. Combining this relation and \eqref{K2-K1-Tr}, we thus obtain the following new iteration complexity for finding an approximate
solution of \eqref{self-concord} by the aforementioned proximal Newton
method.

\begin{theorem} \label{Tr-complexity}
Let $x^0 \in \dom(F)$ and $\epsilon>0$ be given. Suppose that Assumption
\ref{assump2} holds. Then the proximal Newton method \cite[Algorithm 1]{TrKyCe15} for solving problem \eqref{self-concord}  requires at most
\[
\min\left\{
\left\lceil \left[\frac{\log(\omega(0.2))-\log\delta_0}{\log\left(1-\frac{\hc^2  (1- \omega_*^{-1}(\delta_0))} {(1+\hc  \omega_*^{-1}(\delta_0)}\right)}\right]_+\right\rceil, \left\lceil \frac{\delta_0}{\omega(0.2)}\right\rceil  \right\} +  \left\lceil 1.5\log\log\frac{0.28}{\epsilon} \right\rceil
\]
iterations for finding some $x^k$ satisfying $\tlambda(x^k) \le \epsilon$, where $\delta_0=F(x^0)-F^*$, $\hc  = c_3 \sqrt{\sigma_f}$, and $\sigma_f$ and $c_3$ are given in \eqref{sigma-f} and Assumption \ref{assump2}, respectively.
\end{theorem}

{\bf Remark:} Suppose that $g$ is Lipschitz differentiable in $\cS(x^0)$ with a
Lipschitz constant $L_g \ge 0$. It follows from Proposition \ref{suff-cond-assump2} that
Assumption \ref{assump2} holds with $c_3 = \sqrt{\sigma_f}/(L_f+L_g)$,
where $L_f$ is defined in \eqref{L}, and thus Theorem \ref{Tr-complexity} holds
with $\hc  =  \sigma_f/(L_f+L_g)$.

\section{Numerical results}
\label{res}

In this section we conduct numerical experiment to test the performance of RBPDN. In
particular,  we apply RBPDN to solve a regularized logistic regression (RLR) model and
a sparse regularized logistic regression (SRLR) model. We also compare RBPDN with a
randomized block accelerated proximal gradient (RBAPG) method proposed in
\cite{LiLuXi15} on these problems.  All codes are written in MATLAB and all computations are performed on a MacBook Pro running with Mac OS X Lion 10.7.4 and 4GB memory.

For the RLR problem, our goal is to minimize a regularized empirical logistic loss function, particularly, to solve the problem:
\beq \label{logistic}
L^*_{\mu} := \min\limits_{x\in\Re^N} \left\{L_{\mu}(x) := \frac1m \sum^m_{i=1} \log(1+\exp(-y_i \langle w^i,x\rangle)) + \frac{\mu}{2}\|x\|^2\right\}
\eeq
for some $\mu>0$, where $w^i\in\Re^N$ is a sample of $N$ features and $y_i\in\{-1,1\}$ is a binary classification of this sample. This model has
recently been considered in \cite{ZhXi15-1}. Similarly, for the SRLR problem, we aim to solve the problem:
\beq \label{logistic-L1}
L^*_{\gamma,\mu} := \min\limits_{x\in\Re^N} \left\{L_{\gamma,\mu}(x) := \frac1m \sum^m_{i=1} \log(1+\exp(-y_i\langle w^i,x\rangle)) + \frac{\mu}{2}\|x\|^2 + \gamma \|x\|_1 \right\}
\eeq
for some $\mu, \gamma>0$.

In our experiments below, we fix $m=1000$ and set $N=3000, 6000, \ldots, 30000$. For
each pair $(m,N)$, we randomly generate 10 copies of data $\{(w^i,y_i)\}^m_{i=1}$
independently. In each copy,  the elements of $w^i$ are generated according to the
standard uniform distribution on the open interval $(0,1)$ and $y_i$ is generated
according to the distribution $\P(\xi=-1)=\P(\xi=1)=1/2$. As in \cite{ZhXi15-1}, we
normalize  the data so that $\|w^i\|=1$ for all $i=1,\ldots,m$, and set the regularization
parameters $\mu=10^{-5}$ and $\gamma=10^{-4}$.

We now apply RBPDN and RBAPG to solve problem \eqref{logistic}. For both methods, the decision variable $x\in\Re^N$ is divided into 10 blocks sequentially and equally. At each iteration $k$, they pick a block $\i$ uniformly at random. For RBPDN, it needs to find a search direction $d_{\i}(x^k)$ satisfying
 \eqref{d-cond1} and \eqref{d-cond2} with $f=L_{\mu}$ and $g=0$,
 that is,
\beqa
 \nabla^2_{\i\i} L_{\mu}(x^k) d_\i(x^k)+\nabla_{\i} L_{\mu}(x^k) + v_\i = 0, \label{v-subgrad1}\\
\sqrt{\langle v_\i, (\nabla^2_{\i\i} L_{\mu}(x^k))^{-1} v_\i \rangle} \le \eta \sqrt{\langle d_\i(x^k), \nabla^2_{\i\i} L_{\mu}(x^k) d_\i(x^k)\rangle} \label{vi1}
\eeqa
for some $\eta\in[0, 1/4]$. To obtain such a $d_{\i}(x^k)$, we apply conjugate gradient
method to solve the equation
\[
\nabla^2_{\i\i} L_{\mu}(x^k) d_\i = -\nabla_{\i} L_{\mu}(x^k)
\]
until an approximate solution $d_\i$ satisfying
\beq \label{d-error}
\|\nabla^2_{\i\i} L_{\mu}(x^k) d_\i +\nabla_{\i} L_{\mu}(x^k)\| \le  \frac14\sqrt{\mu \langle d_\i, \nabla^2_{\i\i} L_{\mu}(x^k) d_\i \rangle}.
\eeq
is found and then set $d_{\i}(x^k)=d_\i$. Notice from \eqref{logistic} that $\nabla^2_{\i\i} L_{\mu}(x^k) \succeq \mu I$. In view of this, one can verify that such $d_{\i}(x^k)$ satisfies
\eqref{v-subgrad1} and \eqref{vi1} with $\eta=1/4$.
In addition, we choose $x^0=0$ for both methods and terminate them once  the duality
gap is below $10^{-3}$. More specifically,  one can easily derive a dual of problem \eqref{logistic} given by
\[
\max\limits_{s\in\Re^m} \left\{D_{\mu}(s) := -\frac1m \sum^m_{i=1} \log(1-ms_i) - \frac{1}{2\mu}\left\|\sum^m_{i=1}s_iy_iw^i\right\|^2 - \sum^m_{i=1} s_i \log\left(\frac{ms_i}{1-ms_i}\right)\right\}.
\]
Let $\{x^k\}$ be a sequence of approximate solutions to problem \eqref{logistic} generated by RBPDN or RBAPG and $s^k \in \Re^m$ the associated dual
sequence defined as follows:
\beq \label{sk}
s^k_i = \frac{\exp(-y_i\langle w^i,x^k\rangle)}{m(1+\exp(-y_i\langle w^i,x^k\rangle))}, \quad\quad i=1,\ldots, m.
\eeq
We use $L_{\mu}(x^k)-D_{\mu}(s^k) \le 10^{-3}$ as the termination
criterion for  RBPDN or RBAPG, which is checked once every 10 iterations.

The computational results averaged over the 10 copies of data generated above
are presented in Table \ref{res1}. In detail, the problem size $N$ is listed in the first  column. The average number of iterations (upon round off) for RBPDN and RBAPG
are given in the next two columns.  The average CPU time (in seconds) for these methods
are presented in columns four and five, and the average objective function value of
\eqref{logistic} obtained by them are given in the last two columns. One can observe that both methods are comparable in terms of objective values, but RBPDN
substantially outperforms RBAPG  in terms of CPU time.

In the next experiment, we apply RBPDN and RBAPG to solve problem \eqref{logistic-L1}. Same as above,  the decision variable $x\in\Re^N$ is divided into 10 blocks sequentially and equally. At each iteration $k$, they pick a block $\i$ uniformly at random. For RBPDN, it needs to compute a search direction $d_{\i}(x^k)$ satisfying
\eqref{d-cond1} and \eqref{d-cond2} with $f=L_{\gamma,\mu}$ and $g=\gamma \|\cdot\|_1$,
 that is,
\beqa
- v_\i \in  \nabla^2_{\i\i} L_{\mu}(x^k) d_\i(x^k)+\nabla_{\i} L_{\mu}(x^k) + \gamma \partial(\|x^k_\i+d_\i(x^k)\|_1), \label{v-subgrad-3}\\
\sqrt{\langle v_\i, (\nabla^2_{\i\i} L_{\mu}(x^k))^{-1} v_\i \rangle} \le \eta \sqrt{\langle d_\i(x^k), \nabla^2_{\i\i} L_{\mu}(x^k) d_\i(x^k)\rangle} \label{vi-3}
\eeqa
for some $\eta\in[0, 1/4]$. To obtain such a $d_{\i}(x^k)$, we apply FISTA \cite{BeTe09} to solve the problem
\[
\min\limits_{d_\i} \left\{\frac12 \langle d_\i,  \nabla^2_{\i\i} L_{\mu}(x^k) d_\i \rangle + \langle \nabla_{\i} L_{\mu}(x^k), d_\i \rangle + \gamma \|x^k_\i + d_\i\|_1\right\}
\]
until an approximate solution $d_\i$ satisfying \eqref{d-error} and \eqref{v-subgrad-3}  is found and then set $d_{\i}(x^k)=d_\i$. By the same argument as above, one can see
that such $d_{\i}(x^k)$ also satisfies  \eqref{vi-3} with $\eta=1/4$. In addition, we choose $x^0=0$ for both methods and terminate them  the duality
gap is below $10^{-3}$. More specifically, one can easily derive a dual of problem \eqref{logistic-L1} as follows:
\[
\max\limits_{s\in\Re^m} \left\{
\ba{l}
D_{\gamma,\mu}(s) :=  -\frac1m \sum^m_{i=1} \log(1-ms_i) + \frac{\mu}{2}\|h(s)\|^2 + \gamma \|\theta(s)\|_1 - \sum^m_{i=1} s_i \log\left(\frac{ms_i}{1-ms_i}\right) \\ [6pt]
\quad\quad\quad\quad\quad - \langle \sum^m_{i=1}s_iy_iw^i, h(s) \rangle
\ea
\right\},
\]
where
\[
h(s) := \arg\min\limits_{h\in\Re^n}\left\{ \frac{\mu}{2}\|h\|^2  -  \langle \sum^m_{i=1}s_iy_iw^i, h\rangle+ \gamma \|h\|_1\right\}, \quad\quad \forall s \in \Re^m.
\]
Let $\{x^k\}$ be a sequence of approximate solutions to problem \eqref{logistic-L1} generated by RBPDN or RBAPG and $s^k \in \Re^m$ the associated dual
sequence defined as in \eqref{sk}. We use
$L_{\gamma,\mu}(x^k)-D_{\gamma,\mu}(s^k) \le 10^{-3}$ as the termination
criterion for  RBPDN or RBAPG, which is checked once every 10 iterations.

The computational results averaged over the 10 copies of data generated above
are presented in Table \ref{res2}, which is similar to Table \ref{res1} except that
it has two additional columns displaying the average cardinality (upon round off) of
the solutions obtained by RBPDN and RBAPG. We can observe that both methods are comparable in terms of objective values, but RBPDN substantially outperforms RBAPG
 in terms of CPU time and the sparsity of solutions.

\begin{table}[t!]
\caption{Comparison on RBPDN and RBAPG for solving \eqref{logistic}}
\centering
\label{res1}
\begin{tabular}{|c||cc||cc||cc|}
\hline
\multicolumn{1}{|c||}{Problem} &  \multicolumn{2}{c||}{Iteration} &  \multicolumn{2}{c||}{CPU Time} &  \multicolumn{2}{c|}{Objective Value}   \\
 \multicolumn{1}{|c||}{N}
& \multicolumn{1}{c}{\sc RBPDN} & \multicolumn{1}{c||}{\sc RBAPG}
& \multicolumn{1}{c}{\sc RBPDN} & \multicolumn{1}{c||}{\sc RBAPG}
& \multicolumn{1}{c}{\sc RBPDN} & \multicolumn{1}{c|}{\sc RBAPG}   \\
\hline
3000 & 111 & 2837 & 0.13 & 2.01 & 0.2300 & 0.2298 \\
6000 & 53 & 2756 & 0.12 & 3.61 & 0.2142 & 0.2141 \\
9000 & 56 & 2339 & 0.22 & 5.80 & 0.2092 & 0.2092 \\
12000 & 52 & 2083 & 0.32 & 7.64 & 0.2079 & 0.2078 \\
15000 & 48 & 2084 & 0.40 & 10.33 & 0.2069 & 0.2069 \\
18000 & 59 & 1881 & 0.59 & 9.23 & 0.2058 & 0.2059 \\
21000 & 46 & 1866 & 0.55 & 10.28 & 0.2050 & 0.2050 \\
24000 & 53 & 1854 & 0.72 & 11.33 & 0.2050 & 0.2050 \\
27000 & 54 & 1848 & 0.82 & 12.38 & 0.2045 & 0.2044 \\
30000 & 51 & 1924 & 0.87 & 13.87 & 0.2043 & 0.2043 \\
\hline
\end{tabular}
\end{table}

\begin{table}[t!]
\caption{Comparison on RBPDN and RBAPG for solving  \eqref{logistic-L1}}
\centering
\label{res2}
\begin{center}
\begin{small}
\begin{tabular}{|c||cc||cc||cc||cc|}
\hline
\multicolumn{1}{|c||}{Problem} &  \multicolumn{2}{c||}{Iteration} &  \multicolumn{2}{c||}{CPU Time} &  \multicolumn{2}{c||}{Objective Value}  & \multicolumn{2}{c|}{Cardinality}  \\
 \multicolumn{1}{|c||}{N}
& \multicolumn{1}{c}{\sc RBPDN} & \multicolumn{1}{c||}{\sc RBAPG}
& \multicolumn{1}{c}{\sc RBPDN} & \multicolumn{1}{c||}{\sc RBAPG}
& \multicolumn{1}{c}{\sc RBPDN} & \multicolumn{1}{c||}{\sc RBAPG}
& \multicolumn{1}{c}{\sc RBPDN} & \multicolumn{1}{c|}{\sc RBAPG}   \\
\hline
3000 & 2233 & 6126 & 5.44 & 3.19 & 0.5529 & 0.5532 & 749 & 1705 \\
6000 & 1003 & 6239 & 3.82 & 4.74 & 0.5941 & 0.5943 & 840 & 2372 \\
9000 & 626 & 6174 & 3.17 & 6.39 & 0.6210 & 0.6211 & 857 & 3000 \\
12000 & 408 & 5985 & 2.63 & 7.70 & 0.6398 & 0.6400 & 852 & 3108 \\
15000 & 294 & 5762 & 2.30 & 9.06 & 0.6521 & 0.6523 & 815 & 3340 \\
18000 & 272 & 5476 & 2.50 & 10.26 & 0.6616 & 0.6618 & 748 & 3237 \\
21000 & 208 & 5287 & 2.26 & 11.49 & 0.6693 & 0.6694 & 698 & 3173 \\
24000 & 186 & 5146 & 2.31 & 12.76 & 0.6748 & 0.6748 & 650 & 3334 \\
27000 & 180 & 5059 & 2.78 & 14.37 & 0.6790 & 0.6791 & 571 & 4157 \\
30000 & 153 & 4942 & 2.74 & 15.80 & 0.6824 & 0.6824 & 527 & 4312 \\
\hline
\end{tabular}
\end{small}
\end{center}
\end{table}


\begin{thebibliography}{10}

\bibitem{BeTe09}
A.~Beck and M.~Teboulle.
\newblock A fast iterative shrinkage-thresholding algorithm for linear inverse problems.
\newblock {\em SIAM J. Imaging Sciences}, 2(1):183--202, 2009.

\bibitem{BeTe13}
A. Beck and L. Tetruashvili.
\newblock On the convergence of block coordinate descent type methods.
\newblock {\em SIAM J. Optim.}, 13(4):2037--2060, 2013.

\bibitem{BeFa12}
S.~Becker and M.~J.~Fadili.
\newblock A quasi-Newton proximal splitting method.
\newblock In Proceedings  of Neutral Information Processing Systems Foundation, 2012.


\bibitem{Byrd}
R.~H.~Byrd, J.~Nocedal, and S.~Solntsev.
\newblock An algorithm for quadratic $\ell_1$-regularized optimization with a flexible active-set strategy.
\newblock {\em Optim. Method Softw.}, 30(6):1213--1237, 2015.

\bibitem{ChHsLi08}
K.-W. Chang, C.-J. Hsieh, and C.-J. Lin.
\newblock Coordinate descent method for large-scale $\ell_2$-loss linear support vector machines.
\newblock {\em J. Mach. Learn. Res.}, 9:1369--1398, 2008.

\bibitem{DaBaEl08}
A. d'Aspremont, O. Banerjee, and L. El Ghaoui.
\newblock First-order methods for sparse covariance
selection.
\newblock {\em SIAM J. Matrix Anal. Appl.}, 30:56--66, 2008.

\bibitem{FeRi15}
O. Fercoq and P. Richt\'{a}rik.
\newblock Accelerated, parallel and proximal coordinate descent.
\newblock {\em SIAM J. Optim.}, 25(4): 1997--2023, 2015.

\bibitem{HsSuDhRa11}
C. J. Hsieh, M.A. Sustik, I.S. Dhillon, and P. Ravikumar.
\newblock Sparse inverse covariance matrix estimation using quadratic approximation.
\newblock {\em Advances  in  Neutral  Information  Processing Systems (NIPS)}, 24:1--18, 2011.

\bibitem{FrHaTi08}
J. Friedman, T. Hastie, and R. Tibshirani.
\newblock Sparse inverse covariance estimation with the
graphical lasso.
\newblock {\em Biostatistics}, 9:432--441, 2008.

\bibitem{HaYiZh07}
E.~T.~Hale, W.~Yin, and Y.~Zhang.
\newblock Fixed-point continuation applied to compressed sensing: Implementation and numerical experiments.
\newblock {\em J. Comput. Math}, 28(2):170-194, 2010.

\bibitem{HoWa}
M. Hong, X. Wang, M. Razaviyayn, and Z. Q. Luo.
\newblock Iteration complexity analysis of block coordinate descent methods.
\newblock arXiv:1310.6957.


\bibitem{LeSuSa14}
J.D. Lee, Y. Sun, and M.A. Saunders.
\newblock Proximal Newton-type methods for minimizing composite functions.
\newblock {\em SIAM J. Optim.}, 24(3), 1420--1443, 2014.

\bibitem{LeSi13}
Y. T. Lee and A. Sidford.
\newblock Efficient accelerated coordinate descent methods and faster algorithms
for solving linear systems.
\newblock In Proceedings of IEEE 54th Annual Symposium on Founda-
tions of Computer Science (FOCS), pages 147--156, Berkeley, CA, October 2013. Full version at arXiv:1305.1922.

\bibitem{LeLe10}
D. Leventhal and A. S. Lewis.
\newblock Randomized methods for linear constraints: convergence rates
and conditioning.
\newblock {\em Mathematics of Operations Research}, 35(3):641--654, 2010.

\bibitem{LiLuXi15}
Q.~Lin, Z.~Lu, and L.~Xiao.
\newblock An accelerated randomized proximal coordinate gradient method and its application to regularized empirical risk minimization.
\newblock {\em SIAM J. Optim.}, 25(4): 2244--2273, 2015.

\bibitem{LiWr14}
J. Liu, S. J. Wright, C. Re, V. Bittorf, and S. Sridhar.
\newblock An asynchronous parallel stochastic
coordinate descent algorithm.
\newblock {\em JMLR W $\&$ CP}, 32(1):469--477, 2014.

\bibitem{Lu10}
Z. Lu.
\newblock Adaptive first-order methods for general sparse inverse covariance selection.
\newblock {\em SIAM J. Matrix Anal. Appl.}, 31(4):2000--2016, 2010.

\bibitem{LuCh15}
Z. Lu and X. Chen.
\newblock Generalized conjugate gradient methods for $\ell_1$ regularized
convex quadratic programming with finite convergence.
\newblock  arXiv:1511.07837, 2015.

\bibitem{LuXi13}
Z. Lu and L. Xiao.
\newblock A randomized nonmonotone block proximal gradient method for a class of structured nonlinear programming.
\newblock arXiv:1306.5918, 2013.

\bibitem{LuXi15}
Z. Lu and L. Xiao.
\newblock On the complexity analysis of randomized block-coordinate descent methods.
\newblock {\em Math.  Program.}, 152(1-2):615--642, 2015.

\bibitem{Ulbrich}
A.~Milzarek and M.~Ulbrich.
\newblock A semismooth Newton method with multidimensional filter globalization for $l_1$-optimization.
\newblock{\em  SIAM J. Optim.}, 24: 298--333, 2014.

\bibitem{Ne04}
Y.~Nesterov.
\newblock Introductory Lectures on Convex Optimization: A Basic Course.
\newblock Kluwer, Boston, 2004.

\bibitem{Ne07}
Y. Nesterov.
\newblock Gradient methods for minimizing composite objective function.
\newblock {\em Math.  Program.}, 140(1):125--161, 2013.

\bibitem{Ne12}
Y. Nesterov.
\newblock Efficiency of coordinate descent methods on huge-scale optimization problems.
\newblock {\em SIAM J. Optim.}, 22(2):341--362, 2012.

\bibitem{NeNe94}
Y.~Nesterov and A.~Nemirovski.
\newblock Interior Point Polynomial Time Methods in Convex Programming.
\newblock SIAM, Philadelphia, 1994.

\bibitem{PaNe15}
A. Patrascu and I. Necoara.
\newblock Efficient random coordinate descent algorithms for large-scale
structured nonconvex optimization.
\newblock {\em Journal of Global Optimization}, 61(1):19--46, 2015.

\bibitem{QuRiTaFe16}
Z. Qu, P. Richt\'{a}rik, M. Tak\'{a}\v{c}, and O. Fercoq.
\newblock SDNA: stochastic dual Newton ascent for empirical risk minimization.
\newblock {\em Proceedings of The 33rd International Conference on Machine Learning}, 1823--1832, 2016.

\bibitem{RiTa14}
P. Richt\'{a}rik and M. Tak\'{a}\v{c}.
\newblock Iteration complexity of randomized block-coordinate descent
methods for minimizing a composite function.
\newblock {\em Math.  Program.}, 144(1):1--38, 2014.

\bibitem{ShZh13}
S. Shalev-Shwartz and T. Zhang.
\newblock Stochastic dual coordinate ascent methods for regularized
loss minimization.
\newblock {\em J. Mach. Learn. Res.} 14:567--599, 2013.

\bibitem{TrKyCe14}
Q.~Tran-Dinh, A.~Kyrillidis, and V.~Cevher.
\newblock An inexact proximal path-following algorithm for constrained convex minimization.
\newblock {\em SIAM J. Optim.}, 24(4): 1718--1745, 2014.

\bibitem{TrKyCe15}
Q.~Tran-Dinh, A.~Kyrillidis, and V.~Cevher.
\newblock Composite self-concordant minimization.
\newblock {\em J. Mach. Learn. Res.}, 16: 371--416, 2015.

\bibitem{TsYu09}
P. Tseng and S. Yun.
\newblock A coordinate gradient descent method for nonsmooth separable minimization.
\newblock {\em Math.  Program.}, 117:387--423, 2009.

\bibitem{VaFr08}
E.~Van Den Berg and M.~P.~Friedlander.
\newblock Probing the Pareto frontier for basis pursuit solutions.
\newblock {\em SIAM J. Sci. Comp.}, 31(2):890-912, 2008.

\bibitem{WeGoSc12}
Z. Wen, D. Goldfarb, and K. Scheinberg.
\newblock Block coordinate descent methods for semidefinite
programming.
\newblock In M. F. Anjos and J. B. Lasserre, editors, Handbook on Semidefinite, Cone
and Polynomial Optimization: Theory, Algorithms, Software and Applications, volume 166,
pages 533--564. Springer, 2012.

\bibitem{Wr12}
S. J. Wright.
\newblock Accelerated block-coordinate relaxation for regularized optimization.
\newblock {\em SIAM J. Optim.} 22:159--186, 2012.

 \bibitem{WrNoFi09}
S.~J.~Wright, R.~Nowak, and M.~A.~T.~Figueiredo.
\newblock Sparse reconstruction by separable approximation.
\newblock {\em IEEE T. Signal Proces.}, 57:2479--2493, 2009.

\bibitem{XiZh13}
L.~Xiao and T.~Zhang.
\newblock A proximal-gradient homotopy method for the sparse least-squares problem.
\newblock {\em SIAM J. Optim.}, 23(2):1062--1091, 2013.

\bibitem{JYang}
J.~Yang and Y.~Zhang.
\newblock Alternating direction algorithm for $\ell_1$-problems in compressive sensing.
\newblock {\em SIAM J. Sci. Comput.}, 33:250--278, 2010.

\bibitem{YuLi07}
M. Yuan and Y. Lin.
\newblock Model selection and estimation in the Gaussian graphical model.
\newblock {\em Biometrika}, 94:19--35, 2007.

\bibitem{ZhXi15-1}
Y.~Zhang and L.~Xiao.
\newblock Communication-efficient distributed optimization of self-concordant empirical loss.
\newblock arXiv:1501.00263, January 2015.

\bibitem{ZhXi15-2}
Y.~Zhang and L.~Xiao.
\newblock DiSCO: communication-efficient distributed optimization of self-concordant loss.
\newblock International Conference on Machine Learning (ICML 2015).

\end{thebibliography}
\end{document}